\documentclass[reqno, 10pt]{amsart}
\usepackage{color}
\usepackage{amsmath}
\usepackage{amsfonts}
\usepackage{amssymb}
\usepackage{amsthm}
\usepackage{newlfont}
\usepackage{graphicx}

\newtheorem{thm}{Theorem}[section]
\newtheorem{cor}[thm]{Corollary}
\newtheorem{lem}[thm]{Lemma}

\theoremstyle{definition}
\newtheorem{defn}[thm]{Definition}
\theoremstyle{remark}
\newtheorem{rem}[thm]{Remark}

\numberwithin{equation}{section}
\numberwithin{thm}{section}
\newcommand{\norm}[1]{\left\Vert#1\right\Vert}
\newcommand{\abs}[1]{\left\vert#1\right\vert}

\newcommand{\lsm}{\lesssim}

\newcommand{\R}{{\mathbb{R}}}

\def\p{{\frac{4}{d-2}}}
\def\wavekernel{\frac{\sin((t-s)\sqrt{-\Delta})}{\sqrt{-\Delta}}}

\newcommand{\ed}{\end {document}}

\newcounter{smalllist}

\title[energy critical wave equations]
{stability and unconditional uniqueness of solutions for energy
critical wave equations in high dimensions}

\author[A. Bulut]{Aynur Bulut}
\address{Department of Mathematics, University of Texas at Austin}
\email{abulut@math.utexas.edu}

\author[M. Czubak]{Magdalena Czubak}
\address{Department of Mathematics, University of Toronto}
\email{czubak@math.toronto.edu}

\author[D. Li]{Dong Li}
\address{Department of Mathematics, University of Iowa, 14 MacLean Hall, Iowa City, IA, 52240}%
\email{mpdongli@gmail.com}

\author[N. Pavlovi{\' c}]{Nata{\u s}a Pavlovi{\' c}}
\address{Department of Mathematics, University of Texas at Austin}
\email{natasa@math.utexas.edu}

\author[X. Zhang]{Xiaoyi Zhang}
\address{Department of Mathematics, University of Iowa, 14 MacLean Hall, Iowa City, IA, 52240 and
Chinese Academy of Science, Beijing}%
\email{zh.xiaoyi@gmail.com}

\begin{document}
\maketitle

\begin{abstract}
In this paper we establish a complete local theory for the
energy-critical nonlinear wave equation (NLW) in high dimensions
${\mathbb R} \times {\mathbb R}^d$ with $d \geq 6$. We prove the
stability of solutions under the weak condition that the
perturbation of the linear flow is small in certain space-time
norms. As a by-product of our stability analysis, we also prove
local well-posedness of solutions for which we only assume the
smallness of the linear evolution. These results provide essential
technical tools that can be applied towards obtaining the extension
to high dimensions of the  analysis of Kenig and Merle \cite{keme06}
of  the dynamics of the focusing (NLW) below the energy threshold.
By employing refined paraproduct estimates we also prove
unconditional uniqueness of solutions for $d\ge 5$ in the natural
energy class. This extends an earlier result by Planchon \cite{P04}.
\end{abstract}

\section{Introduction}

We consider the Cauchy problem for the energy critical nonlinear wave equation

\begin{align*} (\mbox{NLW})
\qquad\left\lbrace\begin{array}{l}u_{tt}-\Delta u=  F(u), \\
u(0, x) = u_0(x), \\
\partial_t u(0, x) = u_1(x), \end{array}\right.
\end{align*}
where $u(t,x)$ is a real valued function defined on  ${\mathbb R} \times {\mathbb R}^d$
for $d \geq 6$, and $u_0 \in \dot{H}^1({\mathbb R}^d)$, $u_1 \in L^2({\mathbb R}^d)$.
Moreover, the nonlinearity is of a power type given by
$$F(u) = \mu |u|^{\frac{4}{d-2}}u,$$
and $\mu \in \{-1, 1\}$.
We note that $\mu = -1$ corresponds to the defocusing problem, while
$\mu = 1$ corresponds to the focusing problem.

The energy for the (NLW) is given by
\begin{equation} \label{energy}
E\left(u(t), \partial_t u(t) \right)
= \frac{1}{2} \|\partial_t u(t, \cdot) \|_{L^2}^2 + \frac{1}{2} \| \nabla u(t, \cdot) \|_{L^2}^2
- \mu \frac{d-2}{2d} \|u(t, \cdot)\|_{L^\frac{2d}{d-2}}^{\frac{2d}{d-2}},
\end{equation}
and it is conserved in time.
Also we remark that if $u(t,x)$ is a solution to (NLW), then $u_{\lambda}(t,x)$ defined
via
$$ u_{\lambda}(t,x)
   = \frac{1}{\lambda^{\frac{d-2}{2}}} \,
   u(\frac{t}{\lambda}, \frac{x}{\lambda})
$$
is also a solution to (NLW). Since the above scaling leaves the energy
invariant, the (NLW) problem is referred to as ``energy critical''.

Local well-posedness for the Cauchy problem (NLW) has been studied in
many papers (see, e.g. \cite{pe84, gisove92, liso95, shst94, shst98, so95, ka94, keme06}).
Here we recall a version of the local well-posedness result as presented
in \cite{keme06} (see also \cite{pe84, gisove92, shst94}) which states that
for $d = 3, 4, 5$ and initial data $(u_0, u_1) \in \dot{H}^1 \times L^2$,
$\| (u_0, u_1) \|_{\dot{H}^1 \times L^2} \leq A$, $0 \in I$, there exists $\delta = \delta(A)$
such that if
$$ \| K(t) \left(\, u_0, u_1\, \right) \|_{L_{t,x}^{\frac{2(d+1)}{d-2}}(I\times \mathbb R^d)} < \delta,$$
there exists a unique solution to (NLW) in $I \times {\mathbb R}^d$ such that
$(u, \partial_t u) \in C(I; \dot{H}^1 \times L^2)$ and
$ \| u \|_{L_{t,x}^{\frac{2(d+1)}{d-2}} (I\times \mathbb R^d) } \leq 2\delta$.
Here $K$ denotes the associated linear operator i.e.,
$$ K(t)(u_0, u_1) = \cos\left(t \sqrt{-\Delta}\right) \, u_0 +
(-\Delta)^{-\frac{1}{2}} \sin\left(t \sqrt{-\Delta}\right) \, u_1.$$
The proof of this local well-posedness in dimensions $3\le d\le 5$ is
based on the use of the standard Strichartz estimates. However this
proof does not carry directly to high dimensions $d > 6$. The main
reason for this is that for $d > 6$ the derivative of the
nonlinearity is no longer Lipschitz continuous in the standard
Strichartz space.

A natural question related to the local well-posedness theory is the stability
of solutions. Roughly speaking this amounts to showing the closeness
of the solution and an approximate solution,
which solves a perturbed equation, if the perturbations of the equation
and of the initial data are small in a certain sense.
More precisely, let
$\tilde u:\, I\times \mathbb R^d\to \mathbb R$ be an approximate solution which solves
the perturbed NLW:
\begin{align*}
\begin{cases}
\tilde u_{tt} - \Delta \tilde u =  F(\tilde u) + e, \\
\tilde u(t_0,x) = \tilde u_0(x), \\
\partial_t \tilde u(t_0,x)=\tilde u_1(x).
\end{cases}
\end{align*}
Assume the perturbation $e$ is small in a certain norm and the
difference of linear flow measured in terms of scattering size
\begin{align} \label{sta_type}
\left\| K(t) (u_0-\tilde u_0, u_1-\tilde u_1)\right\|_{L_{t,x}^{\frac {2(d+1)}{d-2}}(I\times \R^d)}
\end{align}
is small, then the goal of a typical stability result is to show that
there exists a unique solution $u$ to (NLW) with initial
data $(u_0, u_1)$ such that $u$ and $\tilde u$ stay close on the
whole time interval $I$.  Such a stability result for the (NLW) in $3 \leq d \leq 5$ was obtained in
the work of Kenig and Merle \cite{keme06}.
However the proof
does not carry directly to higher dimensions because the nonlinearity
is no longer Lipschitz in the standard Strichartz space.
This problem was first overcome in the context of the energy-critical nonlinear Schr\"{o}dinger equation (NLS) in \cite{tavi05} in
 $d>6$ by
using certain ``exotic Strichartz" spaces which have same scaling with standard Strichartz space but lower
derivative.\footnote{Actually for smallness condition of type \eqref{sta_type}, exotic Strichartz spaces
are also employed to establish stability theory even in dimensions
$3\le d\le 5$, see, e.g. \cite{kenignotes}. However if instead of \eqref{sta_type} one assumes a stronger condition that
$\| (u_0-\tilde u_0, u_1-\tilde u_1) \|_{\dot H^1\times L^2} $ is small, then the proof of stability
theory can be again carried out by standard Strichartz estimates in dimension $3\le d\le 5$.}
The proof was later simplified in \cite{kivi08} (see Section 3 therein) where
stability is established in Sobolev Strichartz spaces by using fractional chain rule.
In the case of the energy-critical Klein-Gordon equation in high dimension stability
was proved by Nakanishi in \cite{na99}. The main technical difficulty in the context of NLW, besides
choosing the appropriate exotic Strichartz space, is
that in order to show that nonlinearity is Lipschitz continuous in these spaces,
one encounters a problem in establishing H\"older continuity of the nonlinearity
in the standard Strichartz space. This is quite different
from the NLS case since in the latter case one works with the local operator $\nabla$, while in NLW
one has to work with fractional derivatives which are nonlocal.

In the defocusing case the global well-posedness theory was worked out in seminal papers
\cite{st88, gr90, gr92, shst93}. In particular,
Struwe \cite{st88} obtained global well-posedness for the (NLW)
in the radial case when $d=3$. Grillakis \cite{gr90} removed the radial assumption in
$d=3$. The global well-posedness and persistence of regularity was shown for
$3 \leq d \leq 5$ by Grillakis \cite{gr92}, Shatah-Struwe \cite{shst93, shst94, shst98}
and Kapitanski \cite{ka94}. On the other hand, in the focusing case, Levine \cite{le74}
proved that if the initial data $(u_0, u_1) \in \dot{H}^1 \times L^2$ are such that
$E\left(u_0, u_1 \right) < 0$, then the solution must blowup in finite time. Hence,
in the focusing case, the global well-posedness does not hold in general. In particular,
Kenig and Merle in \cite{keme06} presented a detailed study of the focusing case for
$3 \leq d \leq 5$ and showed that depending on the size of the initial data with respect
to the size of the ground state, global well-posedness or blowup occurs.
More precisely, in \cite{keme06}, Kenig and Merle employed sophisticated
``concentrated compactness + rigidity method'', introduced in their work
\cite{keme06s} on the NLS,
to obtain the following dichotomy-type result
under the assumption that $E\left(u_0, u_1 \right) < E\left(W, 0\right)$:
\begin{enumerate}
\item[(i)]
If $\| u_0\|_{\dot{H}^1} < \|W\|_{\dot{H}^1}$, then the global well-posedness holds.
\item[(ii)]
If $\| u_0\|_{\dot{H}^1} >  \|W\|_{\dot{H}^1}$, then a finite time blowup occurs.
\end{enumerate}
Here $W$ denotes the solution to the stationary problem i.e., $W$ satisfies the elliptic
equation
$$ \Delta W + |W|^{\frac{4}{d-2}} W = 0.$$
Many parts of the proof of this dichotomy argument carry out in high dimensions (e.g.
the rigidity theorem is among them). However the local well-posedness as well as a certain stability result
require revisiting in higher dimensions, since as noted above, one has to prove the Lipschitz continuity
of the nonlinearity in the exotic Strichartz spaces and also the H\"older continuity in the standard
Strichartz spaces.

The purpose of this paper is to establish a complete local theory for (NLW) in high dimensions $d\ge 6$
by providing a stability result for the (NLW) in $d\ge 6$ as well as an unconditional uniqueness result in
$\mathbb R \times {\mathbb R}^d$ for $ d \geq 5$. More precisely:

\begin{enumerate}
\item We prove a stability result for  the (NLW) for $d \geq 6$ via
introducing appropriate exotic Strichartz spaces  (in particular,
see the definition of the space $X$ in Section \ref{sec-notation}) and via working in Strichartz spaces of Besov type (see the definition of the space $\dot S^{1}$ in Section \ref{sec-notation}).
In order to prove Lipschitz continuity of nonlinearity in the exotic Strichartz spaces, one usually proves
the H\"older continuity of the nonlinearity in the standard Strichartz space of Sobolev type.  As mentioned above this leads to a technical difficulty which
is different from the NLS case. In the NLS case, the H\"older continuity can be easily established due to
the fact that $\nabla$ is a local operator.  On the other hand, in the NLW case, the standard Strichartz space
involves the fractional derivative which is nonlocal and this causes the technical difficulty to prove
H\"older continuity in the Strichartz space of Sobolev type. We shall circumvent this
difficulty by choosing the working space as Strichartz space of Besov type, space $\dot S^{1},$ and then
transferring the corresponding result to the Sobolev setting (see Remark \ref{rem1}, Lemma
\ref{lem4aa} and Section \ref{sec-perturbation} for more details).
Hence we can prove the main stability result stated in Theorem
\ref{thm_long2} in the pure Sobolev setting
\footnote{In Theorem \ref{thm_long2} we do not
assume smallness in exotic Strichartz spaces, as it was the case
with the stability result for the NLS in \cite{tavi05}.}.

We remark that a direct
side-product of our stability result is continuous dependance of
the data that follows from Theorem \ref{thm_long2} by taking
$e=0$.

Also using the nonlinear estimates that we employ in the stability analysis, we obtain a local in time
existence of solutions to (NLW) and a standard blow-up criterion, see Theorem \ref{lwpthm} and
Lemma \ref{blowup_criterion_lemma} for the precise statements
of these results.

\item  By using paraproduct estimates we
prove unconditional uniqueness of strong solutions to the (NLW) as
stated in Theorem \ref{thm_unconditional}. By unconditional
uniqueness, we mean that for given initial data $(u_0,u_1)$, there
exists at most one solution of (NLW) in the class $C_t \dot H_x^1(I\times \R^d)$. 
In the context of $\dot H^s$ critical NLS, the unconditional uniqueness 
was first established 
by Furioli and Terraneo \cite{FT03} using para-product analysis. 
In the context of energy critical NLW, this problem was first addressed by Planchon \cite{P04}, where the unconditional uniqueness was established in dimensions $d=4, 5$ 
(a review of the unconditional uniqueness for both the NLS and NLW can be found in the paper by Furioli, Planchon and Terraneo \cite{FPT03}).
As a matter of fact, the proof presented in \cite{P04}  can also cover the 6-dimensional case with quadratic nonlinearity  $u^2$. 
The main technical barrier when extending the analysis to high dimensions is that the nonlinearity fails to be $C^2$. Therefore one cannot do Taylor expansion on the nonlinearity to second order as in the low dimensional case, see \cite{P04} for more details. The analysis used in this paper is reminiscent of the one in \cite{P04}; on the other hand, to remove the restriction on the dimension, we need more refined estimates on the nonlinearity.

Interestingly, the proof of unconditional uniqueness also yields a new proof of local well-posedness in high
dimensions $d\ge 5$ (see Remark \ref{rem1043}).

We should also stress that the unconditional uniqueness in $d=3$ is
still open due to the failure of the endpoint Strichartz estimates
except the radial case (see however \cite{MP06} for an interesting result concerning uniqueness of
weak solutions to defocusing NLW in $d=3$ under a local energy
inequality assumption on the light cone).

\end{enumerate}

\vspace{0.3cm}

We remark that the stability result of this paper combined with a modification of
the profile decomposition for the linear wave equation, that was for
$d=3$ obtained by Bahouri and G\'erard \cite{bage99} and extended to high dimensions
$d > 3$ by Bulut \cite{bu09}, implies
that the dichotomy result of Kenig and Merle \cite{keme06}
is valid in all dimensions $d \geq 3$. Hence the stability result of
this paper is a technical tool that can be applied directly to understand
the dynamics of the focusing (NLW) below the energy threshold.

Another application of the stability result obtained in this paper is
in studying the dynamics of the focusing (NLW) at the energy threshold
 $E\left(u_0, u_1 \right) = E\left(W, 0\right)$ in high dimensions.
 Such dynamics were analyzed  by Duyckaerts and Merle \cite{dume07}
for $3 \leq d \leq 5$, and recently by Li and Zhang \cite{lizh09} in high dimensions $d \geq 6$
(see also \cite{dume07s} and \cite{lizh09s} for the NLS case).

\subsection*{Organization of the paper}
In Section \ref{sec-notation}  we introduce the notations and
present various estimates that will be used throughout the paper.
Main results of this paper: the local well-posedness Theorem
\ref{lwpthm}, the unconditional uniqueness Theorem
\ref{thm_unconditional}, the standard blow-up criterion Lemma
\ref{blowup_criterion_lemma} and the stability result Theorem
\ref{thm_long2} are stated in Section \ref{sec-results}. Theorems
\ref{lwpthm}  and  \ref{thm_unconditional} are proved in  Section
\ref{sec-lwp}. In Section \ref{sec-perturbation} we present the
proof of the main stability result, by first presenting a short-term
perturbation result followed by the main long-term perturbation
result.

\subsection*{Acknowledgements}
We are thankful to Carlos Kenig for very useful discussions. We are also grateful to
Fabrice Planchon for some helpful corrections and for pointing out additional
references.
A.B. thanks Institut Henri Poincar\'e (IHP) for supporting her
participation in the Nonlinear Waves and Dispersion trimester during
part of the time that this work was in preparation. The work of D.L.
is supported in part by the NSF grant DMS 0635607, NSF grant
No.~0908032 and the old gold summer fellowship from University of
Iowa. The work of N.P. is supported by NSF grant DMS 0758247 and an
Alfred P. Sloan Research Fellowship. The work of X.Z. is supported
by the NSF grants DMS 0635607, DMS 10601060 and the project 973 in
China.

\section{Notation and Preliminaries} \label{sec-notation}

\subsection{Notations}
In what follows, we write $X\lesssim Y$ or $Y\gtrsim X$ to indicate that there exists a constant $C>0$ such that $X\leq CY$.  We also use the symbol $O(Y)$
to denote any quantity $X$ with the property $|X|\lesssim Y$ and $\nabla$ for the derivative operator in the space variable.

For any time interval $I\subset\mathbb{R}$, we write $L_t^qL_x^r(I\times\mathbb{R}^d)$ to denote the Banach space of functions
 $u:I\times\mathbb{R}^d\rightarrow\mathbb{R}$ with the norm
\begin{align*}
\lVert u\rVert_{L_t^qL_x^r(I\times\mathbb{R}^d)}:=\left(\int_I \left(\int_{\mathbb{R}^d} |u|^rdx\right)^\frac{q}{r} dt\right)^\frac{1}{q}<\infty,
\end{align*}
with the standard definitions when $q$ or $r$ is equal to infinity.  When $q=r$, we abbreviate $L_t^qL_x^q$ as $L_{t,x}^q$.

We define the Fourier transform on $\mathbb{R}^d$ by
\begin{align*}
\hat{f}(\xi):=(2\pi)^{-\frac{d}{2}}\int_{\mathbb{R}^d} e^{-ix\cdot \xi}f(x)dx,
\end{align*}
and, for $s\in\mathbb{R}$, the fractional differentiation operator $|\nabla|^s$ by
\begin{align*}
\widehat{|\nabla|^s f}(\xi):=(4\pi^2|\xi|^2)^{\frac{s}{2}}\hat{f}(\xi),
\end{align*}
which allows us to define the homogeneous Sobolev norm,
\begin{align*}
\lVert f\rVert_{\dot{H}_x^{s,p}}:=\lVert |\nabla|^s f\rVert_{L_x^p(\mathbb{R}^d)}.
\end{align*}
In the case, $p=2$, we abbreviate $\dot{H}_x^{s,2}$ as $\dot{H}_x^s$.

For any constant $C>0$, we define
\begin{align}
\phi_{\le C}(x):=
\phi \bigl( \tfrac{x}{C}\bigr) \quad\textrm{and}\quad \phi_{> C}:=1-\phi_{\le C},
\end{align}
where $\phi\in C^\infty(\R^d)$ is a
radial bump function supported in the ball $\{ x \in \R^d: |x| \leq
\frac{25} {24} \}$ with $\phi(x)=1$ on $\{ x \in \R^d: |x|
\leq 1 \}$.

For each number $ j \in \mathbb Z$, we define the following standard Littlewood-Paley Fourier multipliers
\begin{align*}
\widehat{\Delta_{\leq j} f}(\xi) &:= \phi_{\leq 2^j}(\xi) \hat f(\xi),\\
\widehat{\Delta_{> j} f}(\xi) &:= \phi_{> 2^j}(\xi) \hat f(\xi),\\
\widehat{\Delta_j f}(\xi) &:= (\phi_{\leq 2^j} - \phi_{\leq 2^{j-1}})(\xi) \hat
f(\xi),
\end{align*}
with similar definitions for $\Delta_{<j}$ and $\Delta_{\geq j}$.  Moreover, we define
$$ \Delta_{j < \cdot \leq l} := \Delta_{\leq l} - \Delta_{\leq j} = \sum_{j < m \leq l} \Delta_{m}$$
whenever $j< l$.

We will use Bernstein estimate:
\begin{align}\label{bern}
\norm{\Delta_{j} u}_{L^{q}(\R^d)}\lesssim 2^{j(\tfrac{d}{p}-\tfrac{d}{q})}\norm{\Delta_{j} u}_{L^p(\R^d)},
\end{align}
where $1\leq p \leq q\leq\infty$.

We recall the definition of the homogenous Besov spaces
$\dot{B}^s_{p,q}$ (see for instance \cite{bergh_lof}). For each
$s\in\mathbb{R}$ and $1\leq p\leq \infty$, $1\leq q< \infty$, we
define
\begin{align*}
\lVert u\rVert_{\dot{B}^s_{p,q}}=\left(\sum_{j\in\mathbb{Z}} (2^{sj}\lVert \Delta_j u\rVert_{L^p(\mathbb{R}^d)})^q\right)^\frac{1}{q},
\end{align*}
and
\begin{align*}
\dot{B}^s_{p,q}(\mathbb{R}^d)=\{u\in \mathcal{S}'(\mathbb{R}^d):\lVert u\rVert_{\dot{B}^s_{p,q}(\mathbb{R}^d)}<\infty\}.
\end{align*}
Another equivalent characterization of Besov space will also be used in this paper (see \cite{bergh_lof}).
Namely, for $0<s<1$, $1\le p\le\infty$, $1\le q<\infty$,
\begin{align}
\label{besov_1}\|f\|_{\dot B^s_{p,q}}\sim\biggl(\int_{\R^d}\frac{\|f(x+t)-f(x)\|_p^q}{|t|^{d+sq}}dt\biggr)^{\frac 1q}.
\end{align}
And
\begin{align}
\label{besov_2}\|f\|_{\dot B^s_{p,\infty}}\sim\sup_{t\in \R^d}|t|^{-s}{\|f(x+t)-f(x)\|_p}.
\end{align}

The following lemma is a simple consequence of the definition of Besov norms:
\begin{lem}\label{simplefact}
Let $0<s,\alpha<1$ such that $\frac s{\alpha}<1$. Let $1<p\le\infty$ such that $1< p\alpha\le\infty$. Let $f(z)$ be H\"older continuous of order $\alpha$.  Then,
\begin{align*}
\|f(u)\|_{\dot B^{s}_{p,\infty}}\lsm \|u\|_{\dot B^{\frac s{\alpha}}_{p\alpha,\infty}}^{\alpha}.
\end{align*}
\end{lem}

\begin{proof}
Let $u\in \dot{B}^\frac{s}{\alpha}_{p\alpha,\infty}$ be given.  Then by \eqref{besov_2} we have the inequality,
\begin{align*}
\lVert f(u)\rVert_{\dot{B}^s_{p,\infty}}&\lesssim \sup_{t\in\mathbb{R}^d} \left[ |t|^{-s}\left(\int_{\mathbb{R}^d}\left|f(u(x+t))-f(u(x))\right|^pdx\right)^\frac{1}{p}\right]\\
&\lesssim \sup_{t\in\mathbb{R}^d} \left[|t|^{-s}\left(\int_{\mathbb{R}^d}|u(x+t)-u(x)|^{\alpha p}dx\right)^\frac{1}{p}\right]\\
&=\sup_{t\in\mathbb{R}^d}\left(|t|^{-\frac{s}{\alpha}}\lVert u(x+t)-u(x)\rVert_{L^{p\alpha}}\right)^\alpha\\
&\leq \left(\sup_{t\in\mathbb{R}^d} |t|^{-\frac{s}{\alpha}}\lVert u(x+t)-u(x)\rVert_{L^{p\alpha}}\right)^\alpha\\
&\sim \lVert u\rVert_{\dot{B}^\frac{s}{\alpha}_{p\alpha,\infty}}^\alpha
\end{align*}
where in the second inequality we have used the H\"older continuity of $f$.
\end{proof}

\subsection{Function Spaces}

For dimensions $d\geq 6$ and any time interval $I\subset \R$, we introduce the following norms:
\begin{align}
\nonumber \|u\|_{S(I)}  &= \|u \|_{L_{t,x}^{\frac{2(d+1)}{d-2}}(I\times\R^d)}, \\
\intertext{
\begin{math} \nonumber  \| u \|_{\dot S^1(I)} = \sup \Big\{\|u\|_{L_t^q \dot B^{1-\beta(r)}_{r,2}(I\times\R^d)},\|\partial_t u\|_{L_t^q \dot B^{-\beta(r)}_{r,2}(I\times\R^d)}:(q,r)\,\textrm{wave-admissible}\Big\} \end{math}
}
\nonumber \| u\|_{W(I)} &= \| u\|_{L_t^{\frac{2(d+1)}{d-1}}\dot B^{\frac 12}_{\frac{2(d+1)}{d-1},2} (I\times \R^d) }, \\
\label{norms} \| u\|_{W^\prime (I)} &= \| u\|_{L_t^{\frac{2(d+1)}{d+3}}\dot B^{\frac 12}_{\frac{2(d+1)}{d+3},2} (I\times \R^d) },\\
\nonumber  \| u \|_{X(I)} &=\|u\|_{ L_t^{\frac {d^2+d}{d+2}} \dot H^{\frac 2{d}, \frac {2(d+1)} {d-1} }(I \times \R^d)}, \\
\nonumber \| u\|_{X^\prime(I)}  &= \|u\|_{L_t^{\frac{d^2+d}{3d+2}} \dot H^{\frac 2{d}, \frac {2(d+1)} {d+3} }(I\times \R^d)},\\
\nonumber \|u\|_{Y(I)}  &= \|u\|_{L_t^{\frac{2d^3-7d^2-9d}{d^3-6d^2+7d-2}} \dot H^{\frac{d^2-4d-2}{2d^2-9d}, \frac{4d^3-14d^2-18d}{ 2d^3-11d^2+11d-8}}(I\times \R^d)}.
\end{align}
\begin{rem}\label{rem1}We stress here that the Strichartz space
$\dot S^1$ is defined in terms of Besov spaces. Choosing the working space as a Besov space allows us to
bound the fractional derivative of the difference of the nonlinear term (see Lemma \ref{lem4aa}).
Although Besov spaces are stronger than Sobolev spaces when $p>2$,
Lemma \ref{lem7a} shows that the boundedness of the Sobolev norms of near solutions implies the boundedness of the Besov norms.
Therefore with the help of Lemma \ref{lem7a},
our main theorem (Theorem \ref{thm_long2}) can be proved in the pure Sobolev setting.
\end{rem}

As a consequence of interpolation, we identify the following relationships between the norms defined above in ($\ref{norms}$) and the standard Strichartz spaces.

\begin{lem}[Interpolations] \label{lem3}
Let $d\ge 6$ and $I\subset\mathbb{R}$ be any time interval. Then we have the following inequalities:

\begin{enumerate}
\item[(a)]
\begin{align*}
  \| u\|_X & \lesssim \| u\|_{L_{t,x}^{\frac{2(d+1)}{d-2}}}^{\theta_1 } \cdot
      \|u \|^{1-\theta_1 }_{ L_t^\infty \dot H^{ \frac{2d-4}{d^2-4d-4}, \frac{2d^2-8d-8}{d^2-6d+8} }} \\
  & \lesssim  \| u\|_{S}^{\theta_1}
\cdot \| u\|_{L_t^\infty \dot H^1}^{1-\theta_1},
 \end{align*}
where $\theta_1 = \frac {2d+4}{d^2-2d} $.

\item[(b)]
\begin{align*}
 \| u\|_{S} &\lesssim \|u \|_X^{\theta_2} \cdot
  \| u\|_{L_t^{\frac{2(d+1)}{d-1}} \dot H^{\frac 12,\frac{2(d+1)}{d-1}}}^{1-\theta_2} \\
&\lesssim \|u \|_X^{\theta_2} \cdot
  \| u\|_{W}^{1-\theta_2}
\end{align*}
where $\theta_2 = \frac{d}{d^2-3d-4}$.

\item[(c)]
\begin{align*}
 \| u \|_{L_t^{\frac{4(d+1)}{d-2}} L_x^{\frac{4(d^2+d)}{2d^2-3d-2}}}
&\lesssim
\|u\|_{X}^{\theta_3} \cdot \| u\|_{L_t^{\frac{2(d+1)}{d-1}} \dot H^{\frac 12, \frac{2(d+1)}{d-1}} }^{1-\theta_3} \\
&\lesssim
\|u\|_{X}^{\theta_3} \cdot \| u\|_{W }^{1-\theta_3}
\end{align*}
where $\theta_3 = \frac {d^2}{ 2(d^2-3d-4)}$.

\item[(d)]
\begin{align*}
 \|u \|_{L_t^{\frac {2(d+1)}{d-2}} \dot H^{\frac 12, \frac{2(d^2+d)}{d^2-d+1}}}
& \lesssim \|u \|_X^{\theta_4} \cdot
  \| u
\|^{1-\theta_4}_{Y}
\\
& \lesssim \|u \|_X^{\theta_4} \cdot \|u\|_{\dot S^1}^{1-\theta_4},
\end{align*}
where $\theta_4 = \frac1 {2(d-4)}$.

\item[(e)]
We also have the embedding
\begin{align*}
 \dot S^1 \hookrightarrow L_t^{\frac{d^2+d}{d+2} } \dot H^{\frac{d^2-2d-2}{d^2-d}, \frac{2d^3-2d}{d^3-5d-8}}
\hookrightarrow X.
\end{align*}
\end{enumerate}
\end{lem}

\subsection{Strichartz Estimates}

We state the Strichartz estimates for the wave equation, which we frequently use throughout the paper (see for instance \cite{GinibreVelo}, \cite{KeelTao}, \cite{liso95}).

\begin{lem}[Strichartz] \label{lem1}
 Let the pairs $(q_i,r_i)$, $i=1,2$, satisfy
\begin{align*}
\frac 2 {q_i} &= (d-1) (\frac 12 -\frac 1 {r_i}),\\
2\leq q_i,\, r_i\le \infty, &\quad\textrm{and}\quad (q_i,r_i,d)\ne (2,\infty, 3),
\end{align*}
and let $u$ satisfy
\begin{align*}
\left\lbrace\begin{array}{ll}u_{tt}-\Delta u = f, \\ u(0)=u_0\in \dot{H}^1, \\ u_t(0)=u_1\in L^2.\end{array}\right.
\end{align*}
Then
\begin{align*}
 \| u\|_{L_t^{q_1} \dot B^{1-\beta(r_1)}_{r_1,2}} + \| \partial_t u \|_{L_t^{q_1} \dot B^{-\beta(r_1)}_{r_1,2}}
 \lesssim
\|u_0\|_{\dot H^1} + \|u_1\|_{L^2} + \|f\|_{L_t^{q^\prime_2} \dot B^{\beta(r_2)}_{r^\prime_2,2}},
\end{align*}
where $\beta(r_i) = \frac {d+1} 2 (\frac 12 -\frac 1 {r_i})$ and $\frac{1}{q_2}+\frac{1}{q'_2}=\frac{1}{r_2}+\frac{1}{r'_2}=1$.
\end{lem}

Now, we record the following decay estimate (see \cite{GinibreVelo}).
\begin{lem}[Decay estimate for $\frac {\sin (t \sqrt{-\Delta}) }{\sqrt{-\Delta}}$] \label{lem2}
 We have
\begin{align*}
 \left\| \frac {\sin (t\sqrt{-\Delta})} {\sqrt{-\Delta}} f \right\|_{\dot B^{1-\beta(r)}_{r,2}}
\lesssim |t|^{-\gamma(r)} \| f\|_{\dot B^{\beta(r)}_{r^\prime,2}},
\end{align*}
where $0\le \gamma(r) = (d-1)(\frac 12 -\frac 1r) \le 1$.
\end{lem}

As a consequence of Lemma \ref{lem2}, we prove a Strichartz estimate establishing a connection between the spaces $X(I)$ and $X'(I)$.  This estimate will be essential for obtaining appropriate estimates of the nonlinear term.

\begin{lem}[Exotic Strichartz in $X$] \label{lem4}
Let $0\in I$ be a time interval. Then,
\begin{align}
 \left \| \int_0^t \frac {\sin((t-\tau) \sqrt{-\Delta} )} { \sqrt{-\Delta}} f(\tau) d\tau \right\|_{X(I)}
\lesssim \|  f \|_{X^\prime (I)}. \label{exotic}
\end{align}

\end{lem}
\begin{proof}
The inequality \eqref{exotic} follows directly from the decay estimate (Lemma \ref{lem2}) and
the Hardy-Littlewood-Sobolev inequality in time.
\end{proof}

\subsection{Nonlinear Estimates}

In many of our arguments, we will require estimates on the nonlinearity.  To obtain these estimates, our main tools will be several facts from
fractional calculus.

\begin{lem}[Fractional Leibniz rule \cite{ChristWeinstein}]
\label{lem8}
Let $s\in (0,1]$ and $1<r,p_1,p_2,q_1,q_2<\infty$ be given such that $\frac{1}{r}=\frac{1}{p_i}+\frac{1}{q_i}$ for $i=1,2$.  Then there exists $C>0$ such that,
\begin{align*}
\lVert \abs{\nabla}^s(fg)\rVert_{L^r}&\leq C\lVert f\rVert_{L^{p_1}}\lVert \abs{\nabla}^sg\rVert_{L^{q_1}}+\lVert \abs{\nabla}^sf\rVert_{L^{p_2}}\lVert g\rVert_{L^{q_2}}.
\end{align*}
\end{lem}

\begin{lem}[$C^1$ fractional chain rule \cite{ChristWeinstein}]
\label{lem5}
Suppose $G\in C^1(\mathbb{C})$, $s\in (0,1]$, and $1<q,q_1,q_2<\infty$ are such that $\frac{1}{q}=\frac{1}{q_1}+\frac{1}{q_2}$.  Then
\begin{align*}
\lVert \abs{\nabla}^sG(u)\rVert_{L^q}\lesssim \lVert G'(u)\rVert_{L^{q_1}}\lVert \abs{\nabla}^su\rVert_{L^{q_2}}.
\end{align*}
\end{lem}
When $G$ fails to be $C^1$, but remains H\"older continuous, we have the following version of the chain rule.
\begin{lem}[$C^\alpha$ fractional chain rule \cite{Visan}]
\label{lem6}
Let G be a H\"older continuous function of order $0<\alpha<1$.  Then for every $0<s<\alpha$, $1<p<\infty$ and $\frac{s}{\alpha}<\sigma<1$, there exists $C>0$ such that
\begin{align*}
\lVert \abs{\nabla}^sG(u)\rVert_{L^p(\mathbb{R}^d)}\leq C\lVert |u|^{\alpha-\frac{s}{\sigma}}\rVert_{L^{p_1}(\mathbb{R}^d)}\lVert \abs{\nabla}^\sigma u\rVert_{L^{\frac{s}{\sigma}p_2}(\mathbb{R}^d)}^\frac{s}{\sigma}
\end{align*}
provided $\frac{1}{p}=\frac{1}{p_1}+\frac{1}{p_2}$ and $(1-\frac{s}{\alpha\sigma})p_1>1$.
\end{lem}

We now prove the following lemma, which is an essential tool in obtaining the H\"older continuity of the nonlinearity in Strichartz spaces of
Besov type (see Section \ref{sec-perturbation} for more details).

\begin{lem} \label{lem4aa}
Let $\frac 1p = \frac 1{p_1} +\frac 1{p_2} = \frac 1 {p_3} + \frac 1{p_4}$. $1<p_i<\infty$, $i=1,2,3,4$.
Assume the function $F\in C^{1,\alpha}(\R,\R)$, $0<\alpha<1$.
Then
\begin{align} \label{es_0}
\| F(u)-F(v) \|_{\dot B^{\frac 12}_{p,2}}
\lesssim \| u-v\|_{\dot B^{\frac 12}_{p_1,2}} \cdot \| |u|^\alpha \|_{p_2}
+ \| |u-v|^\alpha \|_{p_3} \cdot \| v\|_{\dot B^{\frac 12}_{p_4,2}}.
\end{align}
\end{lem}
\begin{proof}
 This is a simple consequence of the definition of Besov space and the H\"older inequality. Recall that
for $0<s<1$, $1<p<\infty$,
\begin{align} \label{es_1}
 \| f\|_{\dot B^s_{p,2} (\R^d)} =
 \left( \int_{\R^d} \frac{ \| f(x+t) -f(x)\|_p^2} {|t|^{d+2s}} dt \right)^{\frac 12}.
\end{align}
By using the fundamental theorem of calculus, we have
\begin{align*}
 &F(u(x+t)) - F(u(x)) \\
&\hspace{0.2in}= (u(x+t)-u(x)) \cdot \int_0^1 F^\prime ( \lambda u(x+t) + (1-\lambda) u(x) ) d\lambda
\end{align*}
and
\begin{align*}
 &F(v(x+t)) - F(v(x)) \\
&\hspace{0.2in}=  (v(x+t)-v(x)) \cdot \int_0^1 F^\prime ( \lambda v(x+t) + (1-\lambda) v(x) ) d\lambda
\end{align*}
Subtracting the above two identities and rearranging terms, we obtain
\begin{align*}
 & F(u(x+t)) - F(v(x+t)) - F(u(x)) +F(v(x)) \\
&\hspace{0.2in}=  ( (u-v)(x+t) - (u-v)(x)) \int_0^1 F^\prime (\lambda u(x+t) + (1-\lambda) u(x)) d\lambda \\
& \hspace{0.4in} + (v(x+t)-v(x)) \int_0^1 ( F^\prime( \lambda u(x+t) + (1-\lambda) u(x))\\
&\hspace{0.4in}-F^\prime ( \lambda v(x+t) + (1-\lambda) v(x)) ) d\lambda.
\end{align*}
Therefore by H\"older continuity of $F^\prime$ and translation invariance of $L^p$ norms in $\R^d$, we get
\begin{align}
 & \| F(u(x+t)) - F(v(x+t)) - F(u(x)) +F(v(x)) \|_p \notag \\
&\hspace{0.2in}\le   \| (u-v)(x+t) - (u-v)(x) \|_{p_1} \cdot \| |u|^\alpha \|_{p_2} \notag \\
& \hspace{0.4in} + \| v(x+t) -v(x) \|_{p_3} \cdot \| |u-v|^\alpha \|_{p_4}, \label{es_2}
\end{align}
where $\frac 1p = \frac 1 {p_1} + \frac 1 {p_2} = \frac 1 {p_3} + \frac 1{p_4}$. Now
clearly \eqref{es_0} follows from \eqref{es_1} and \eqref{es_2}.
\end{proof}

With these estimates in hand, we now prove some further inequalities that will help us to bound the nonlinear term.

\begin{lem}[Nonlinear estimates] \label{lem4a}
We have
\begin{align}
\norm{ F(u)}_{W'(I)}&\lesssim \norm{u}_{X(I)}^{\theta_2\p}\norm{u}_{\dot S^1(I)}^{(1-\theta_2)\p+1}\label{e218d}\\
\norm{ F(u)}_{X'(I)}&\lesssim \norm{u}_{X(I)}^{\theta_2\p+1}\norm{u}_{\dot S^1(I)}^{(1-\theta_2)\p}\label{e218e}\\
 \| |\nabla|^{\frac 2d} F^\prime (u ) \|_{L_t^{\frac{d+1}2} L_x^{\frac{d^3+d^2}{2d^2+2d+2}} (I)}
& \lesssim  \|u\|^{\frac 4d}_{L_t^{\frac{2(d+1)}{d-2}} \dot H^{\frac 12,\frac{2(d^2+d)}{d^2-d+1}} (I)}
\cdot \| u\|_{S (I)}^{\frac 8{d(d-2)} }. \label{e218a} \\
 \| F(u) -F(w) \|_{X^\prime(I)} & \lesssim \| u-w\|_{X(I)} \cdot ( \|u-w\|^{\frac 4{d-2}}_{S(I)} + \|w\|^{\frac 4{d-2}}_{S(I)} )
 \notag\\
\notag&\quad+ \|u-w\|_{X(I)}
\cdot ( \|u-w\|_{S(I)} +  \|w\|_{S(I)} )^{\frac 8{d(d-2)}}\\
& \qquad \cdot ( \|u-w\|_{X(I)}^{\theta_4} \cdot \|u-w\|_{\dot S^1(I)}^{1-\theta_4} \notag\\
& \qquad +\|w\|_{X(I)}^{\theta_4} \cdot \| w\|_{Y(I)}^{1-\theta_4} )^{\frac 4d}
\label{e218b}
\\
  \| F(u) -F(w) \|_{W^\prime(I)} & \lesssim
\| u-w\|_{W(I)} \cdot ( \|u-w\|^{\frac 4{d-2}}_{S(I)} + \|w\|^{\frac 4{d-2}}_{S(I)} ) \notag\\
& \quad + \|u-w\|^{\frac 4{d-2}}_{S(I)} \cdot \|w\|_{W(I)}. \label{e218c}
\end{align}

\end{lem}
\begin{proof}
First \eqref{e218d} and \eqref{e218e} follow from Lemma \ref{lem5}, H\"older in time, Lemma \ref{lem3}, and $\dot S^1\hookrightarrow W(I)$. \eqref{e218a} follows from Lemma \ref{lem6} and H\"older in time. Next we establish
\eqref{e218b}. By the fundamental theorem of calculus,
\begin{align*}
 F(u) - F(w) = (u-w) \int_0^1 F^\prime ( \lambda u + (1-\lambda) w) d\lambda.
\end{align*}
Therefore by Lemma \ref{lem8} and H\"older in time,
\begin{align}
&\| F(u) -F(w) \|_{X^\prime(I)} \notag \\
&\hspace{0.2in}  \lesssim \sup_{0\le \lambda \le 1} \left\|
\| |\nabla|^{\frac 2{d}} (u-w) \|_{L_x^{\frac{2(d+1)}{d-1}}} \cdot
\| F^\prime ( \lambda u + (1-\lambda) w ) \|_{L_x^{\frac {d+1}2}} \right\|_{L_t^{\frac {d^2+d}{3d+2}}}
\label{aa1} \\
&\hspace{0.4in} + \sup_{0\le \lambda \le 1} \left\|
\| |\nabla|^{\frac 2{d}}  F^\prime ( \lambda u + (1-\lambda) w ) \|_{L_x^{\frac{d^3+d^2}{2d^2+2d+2}}}
\cdot
\| u-w \|_{L_x^{\frac{2d^3+2d^2}{d^3-d^2-4d-4} }} \right\|_{L_t^{\frac {d^2+d}{3d+2}}}
\label{aa2}
\end{align}
For \eqref{aa1}, by H\"older we have
\begin{align}
\eqref{aa1} \lesssim \| u-w\|_{X(I)} \cdot ( \|u-w\|^{\frac 4{d-2}}_{S(I)} + \|w\|^{\frac 4{d-2}}_{S(I)} ).
\label{et20a}
\end{align}
Similarly, for \eqref{aa2}, by H\"older, \eqref{e218a}, Sobolev and Lemma \ref{lem3}, we have
\begin{align}
\eqref{aa2} & \lesssim
\sup_{0\le \lambda \le 1} \| |\nabla|^{\frac 2{d}}  F^\prime ( \lambda u + (1-\lambda) w )
\|_{L_t^{\frac{d+1}2} L_x^{\frac{d^3+d^2}{2d^2+2d+2}} (I)}
\cdot
\| u-w \|_{L_t^{\frac{d^2+d}{d+2}} L_x^{\frac{2d^3+2d^2}{d^3-d^2-4d-4} } (I) } \notag \\
& \lesssim \|u-w\|_{X(I)} \cdot ( \|u-w\|_{X(I)}^{\theta_4} \cdot \|u-w\|_{\dot S^1(I)}^{1-\theta_4} +
\|w\|_{X(I)}^{\theta_4} \cdot \| w\|_{Y(I)}^{1-\theta_4} )^{\frac 4d} \notag\\
& \qquad \cdot ( \|u-w\|_{S(I)} +  \|w\|_{S(I)} )^{\frac 8{d(d-2)}}. \label{et20b}
\end{align}
Clearly now \eqref{e218b} follows from \eqref{et20a} and \eqref{et20b}.
Finally, \eqref{e218c} follows directly from Lemma \ref{lem4aa} and H\"older in time.
\end{proof}

\section{Statements of main results}  \label{sec-results}

In this section we state the main results of this paper.
We begin by recalling the definition
of a strong solution to the Cauchy problem (NLW).

\begin{defn}[Strong solution]
We call $u$ a \emph{strong solution} to (NLW) on a time interval $I$ if $u \in C(I, \, \dot H^1)$
and satisfies the Duhamel formula
\begin{align*}
 u(t) = K(t) ( u_0 , u_1) + \int_{0}^t \frac{ \sin ( \sqrt{-\Delta} (t-\tau))} {\sqrt{-\Delta}}
F(u(\tau)) d \tau
\end{align*}
in the sense of tempered distributions for every $t\in I$.
\end{defn}
\begin{rem}
We stress here that the definition of a strong solution only requires the fact that $u \in C(I,\,
\dot H^1)$. In particular, Strichartz space is not involved in the definition of the solution.
\end{rem}

As discussed in the introduction, the local theory for (NLW) has been extensively studied.  We now formulate Theorem \ref{lwpthm} resembling the statement in \cite{keme06}.  The proof combines the ideas from \cite{keme06} with the ideas used in the proof of local existence in \cite{tavi05}.
As a result we obtain the local existence in the space $\dot S^1$ and local well-posedness in $X$.
\begin{thm}
\label{lwpthm}
Let $d\geq 6$, $(u_0,u_1)\in \dot{H}^1(\R^d)\times L^2(\mathbb{R}^d)$, and $I\subset \mathbb{R}$ be an interval with $t_0=0\in I$ such that
\begin{align*}
\lVert (u_0,u_1)\rVert_{\dot{H}^1\times L^2}\leq A.
\end{align*}
Then there exists $\eta=\eta(A)$ such that
\begin{align*}
\lVert K(t)(u_0,u_1)\rVert_{S(I)}<\eta
\end{align*}
implies that there exists a unique solution $u$ to (NLW) with $(u,\partial_t u)\in C(I;\dot{H}^1\times L^2)$, and
\begin{align*}
\lVert u\rVert_{W(I)}+\lVert \partial_t\abs{\nabla}^{-1}u\rVert_{W(I)}&<\infty,\\
\norm{u}_{\dot S^1(I)}&<\infty,\\
\lVert u\rVert_{X(I)}&\leq 2\delta,
\end{align*}
where $\delta =C\eta^{\theta_1}A^{1-\theta_1},$ where $\theta_1$ is as in Lemma \ref{lem3}.
\end{thm}

We prove the unconditional uniqueness of strong solutions as stated in the
following theorem:

\begin{thm}[Unconditional uniqueness of strong solutions] \label{thm_unconditional}
Let $u$, $v$ be two strong solutions of (NLW) on $I$. Suppose $u(t_0)=v(t_0)$, $u_t(t_0)=v_t(t_0)$
for some $t_0 \in I$, then $u(t)=v(t)$, $\forall\, t\in I$.
\end{thm}

We also prove the following lemma which gives the standard blow-up criterion, that was formulated for
the (NLW) in ${\mathbb R} \times {\mathbb R}^d$ for $d=3,4,5$ by Kenig and Merle in
\cite{keme06}. Here we extend  this blow-up criterion to higher dimensions $d \geq 6$
by following the ideas of the proof of the blow-up criterion for the NLS in high dimensions
\cite{tavi05}.

\begin{lem}[Standard Blow-Up Criterion] \label{blowup_criterion_lemma}Let $(u_0,u_1)\in \dot{H}^1(\mathbb{R}^d)\times L^2(\mathbb{R}^d)$ and $u\in C([t_0,T_0],\dot{H}^1)$ be given such that $u$ is a strong solution to (NLW) on $[t_0,T_0]$ and
\begin{align*}
\lVert u\rVert_{S([t_0,T_0])}<\infty.
\end{align*}
Then there exists $\delta=\delta(u_0,u_1)$ such that $u$ extends to a strong solution to (NLW) on $[t_0,T_0+\delta]$.
\end{lem}

The main result of this paper is the following long term perturbation
theorem, the proof of which we present in Section \ref{sec-perturbation}.

\begin{thm}[Long time perturbation, Sobolev version] \label{thm_long2}
Assume $\tilde u$ is a near solution on  $I\times \R^d$
\begin{align*}
 \partial_{tt} \tilde u -\Delta \tilde u= F(\tilde u) +e
\end{align*}
such that

\begin{enumerate}
\item[(a)]
\begin{align}
   \|\tilde u \|_{L_t^\infty \dot H^1(I) } +
 \| \partial_t  \tilde u\|_{L_t^\infty L^2 (I)}
 +\| \tilde u \|_{L_t^{\frac{2(d+1)}{d-1}} \dot H^{\frac 12,
\frac{2(d+1)}{d-1}} (I)} \le E, \label{e133a}
   \end{align}

\item[(b)]
 \begin{align*}
    \|\tilde u_0 -u_0 \|_{\dot H^1} + \| \tilde u_1 -u_1 \|_{L^2} \le E^\prime.
   \end{align*}

\item[(c)]
Smallness:
\begin{align}
 \| K(t) (\tilde u_0 - u_0, \tilde u_1 -u_1) \|_{S(I)} & \le \epsilon  \label{e133b} \\
 \| |\nabla|^{\frac 12} e \|_{L_{t,x}^{\frac{2(d+1)}{d+3}}(I)} \le \epsilon. \label{e133c}
\end{align}
\end{enumerate}

Then there exists $\epsilon_0 = \epsilon_0(d,E^\prime,E)$ such that if
$0<\epsilon<\epsilon_0$, for (NLW) with initial data $(u_0,u_1)$, there exists
a unique solution $u$ on $I\times \R^d$ with the properties
\begin{align}
 \| \tilde u - u\|_{\dot S^1(I)} & \le C(d,E^\prime,E) \cdot E^\prime, \notag\\
\|\tilde  u- u\|_{S(I)} & \le C(d,E^\prime,E) \cdot \epsilon^c. \label{e210}
\end{align}
Here $0<c<1$ is a constant depending only on the dimension $d$.
\end{thm}

\section{Local well-posedness} \label{sec-lwp}

In this section we prove that the Cauchy problem
(NLW) is locally wellposed.  Let $t_0\in\mathbb{R}$ be given.  By time translation invariance, we may assume $t_0=0$.

\subsection{The proof of the local existence Theorem \ref{lwpthm}.}
First we observe that by Lemma \ref{lem3}
\begin{align}
\norm{K(t)(u_0,u_1)}_{X(I)}&\leq C\norm{K(t)(u_0,u_1)}_{S(I)}^{\theta_1}\norm{K(t)(u_0,u_1)}_{L^\infty_t\dot H^1}^{1-\theta_1}\nonumber\\
&\leq C\eta^{\theta_1}A^{1-\theta_1}\label{m0}
\end{align}
Let $$\delta\equiv C\eta^{\theta_1}A^{1-\theta_1}.$$
Next, we define the sequence of iterates by
\begin{align*}
&u^{-1}=0,\\
&u^0=K(t)(u_0,u_1),\\
&u^{n+1}=K(t)(u_0,u_1)+\int_0^t\wavekernel F(u^{n})(s)ds,
\end{align*}
We show that the sequence is bounded in $\dot S^1(I)$ and in $X(I)$.  By \eqref{m0} and the Strichartz inequality, we have
\begin{align}\label{m3}
\norm{u^0}_{X(I)}\leq \delta\quad\mbox{and}\quad\norm{u^0}_{\dot S^1(I)}\leq CA.
\end{align}
Let $a=2\delta$ and $b=2CA$, and suppose for $n\geq 1$
\[
\norm{u^n}_{X(I)}\leq a\quad\mbox{and}\quad \norm{u^n}_{\dot S^1(I)}\leq b.
\]
Then, by the Strichartz inequality and \eqref{e218d}, we obtain
\begin{align*}
\norm{u^{n+1}}_{\dot S^1(I)}&\leq CA+C\lVert u^n\rVert_{X(I)}^{\theta_2\p}\norm{u^n}_{\dot S^1(I)}^{(1-\theta_2)\p}\lVert u^n\rVert_{\dot S^1(I)}\\
&\leq \frac b2+Ca^{\theta_2\p}b^{(1-\theta_2)\p}b\\
&\leq b,
\end{align*}
if we choose $a$ small enough so that $Ca^{\theta_2\p}b^{(1-\theta_2)\p}\leq \frac{1}{2}$.
Similarly, by \eqref{m3}, Lemma $\ref{lem4}$, and \eqref{e218e}, we get
\begin{align*}
\lVert u^{n+1}\rVert_{X(I)}
&\leq \delta+C\norm{u^n}_{X(I)}^{\theta_2\p}\norm{u^n}_{\dot S^1(I)}^{(1-\theta_2)\p}\norm{u^n}_{X(I)}\\
&\leq \frac a2+Ca^{\theta_2\p}b^{(1-\theta_2)\p}a\\
&\leq a,
\end{align*}
assuming that $a$ is chosen such that it satisfies the same smallness condition as above.  Hence, by induction we have
\[
\norm{u^n}_{X(I)}\leq a\quad\mbox{and}\quad \norm{u^n}_{\dot S^1(I)}\leq b,\quad n\geq 0.
\]
Next we show the sequence is Cauchy in $X(I)$.  To that end, we note that applying Lemma $\ref{lem4}$ and \eqref{e218b} allows us to obtain
\begin{align}
&\nonumber \norm{u^{n+1}-u^n}_{ X(I)}\\
&\hspace{0.2in}\lesssim \norm{F(u^n)-F(u^{n-1})}_{X'(I)}\nonumber \\
&\hspace{0.2in}\lesssim \| u^n-u^{n-1}\|_{X(I)} \cdot ( \|u^n-u^{n-1}\|^{\frac 4{d-2}}_{S(I)} + \|u^{n-1}\|^{\frac 4{d-2}}_{S(I)} )\label{m1}\\
&\hspace{0.4in} + \|u^n-u^{n-1}\|_{X(I)} \cdot ( \|u^n-u^{n-1}\|_{X(I)}^{\theta_4} \cdot \|u^n-u^{n-1}\|_{\dot S^1(I)}^{1-\theta_4}\nonumber \\
&\hspace{0.6in}+\|u^{n-1}\|_{X(I)}^{\theta_4} \cdot \| u^{n-1}\|_{Y(I)}^{1-\theta_4} )^{\frac 4d} \nonumber\\
&\hspace{0.4in} \cdot ( \|u^n-u^{n-1}\|_{S(I)} +  \|u^{n-1}\|_{S(I)} )^{\frac 8{d(d-2)}}.\label{m2}
\end{align}
Then by Lemma \ref{lem3} we get
\begin{align*}
\eqref{m1}\lesssim \| u^n-u^{n-1}\|_{X(I)}a^{\theta_2\p}b^{(1-\theta_2)\p},
\end{align*}
and using $\dot S^1(I)\hookrightarrow Y(I)$ we get
\begin{align*}
\eqref{m2}&\lesssim
\|u^n-u^{n-1}\|_{X(I)} \cdot ( \|u^n-u^{n-1}\|_{X(I)}^{\theta_4} \cdot \|u^n-u^{n-1}\|_{\dot S^1(I)}^{1-\theta_4}   \nonumber\\
&\qquad
+\|u^{n-1}\|_{X(I)}^{\theta_4} \cdot \| u^{n-1}\|_{\dot S^1(I)}^{1-\theta_4} )^{\frac 4d}\nonumber \cdot ( \|u^n-u^{n-1}\|_{S(I)} +  \|u^{n-1}\|_{S(I)} )^{\frac 8{d(d-2)}}\\
&\lesssim
\|u^n-u^{n-1}\|_{X(I)} \cdot ( a^{\theta_4} b^{1-\theta_4})^{\frac 4d}\cdot ( a^{\theta_2} b^{1-\theta_2})^{\frac 8{d(d-2)}}.
\end{align*}
It follows that if $a$ is small enough, the sequence converges to $u$ in $X(I)$.  Since $u^n$ are bounded in $\dot S^1$, they are in particular bounded in $W(I)$, which is reflexive, so $u^n$ converge weakly to $u$ in $W(I)$.  Then, by the Strichartz inequality, we conclude $u \in \dot S^1(I)$.  Also standard arguments using the nonlinear estimate \eqref{e218b} and essentially repeating the calculations above show $u$ solves (NLW) as needed.

\subsection{Unconditional uniqueness}

Having proved the existence of solutions stated in Theorem $\ref{lwpthm}$, we now prove Theorem \ref{thm_unconditional}, which gives the unconditional uniqueness of strong solutions.

We first recall the following fact about Besov norms, which can be proved using
basic properties of Littlewood-Paley operators.

\begin{lem}[Equivalence of Besov norms] \label{lem_Besov_def}
 Let $1<p<\infty$ and $1\le q\le \infty$. Let $s>0$. Then
\begin{align*}
 \| f\|_{\dot B^s_{p,q}} \thickapprox \left( \sum_{j\in \mathbb Z}
 \Big( 2^{js} \| \Delta_{\ge j} f \|_{L_x^p} \Big)^q \right)^{\frac 1q},
\end{align*}
and
\begin{align*}
 \| f\|_{\dot B^{-s}_{p,q}} \thickapprox \left( \sum_{j\in \mathbb Z}
\Big( 2^{-js} \| \Delta_{\le j} f \|_{L_x^p} \Big)^q \right)^{\frac 1q}.
\end{align*}
\end{lem}

In our proof of Theorem $\ref{thm_unconditional}$ we will use the following fact regarding paraproducts.  For any two functions $f$ and $g$, we may decompose the product $fg$ into the sum of a low frequency piece and a high frequency piece.  Indeed, by frequency localization, we write
\begin{align}
 fg & = \sum_{j \in \mathbb Z} \Delta_j (fg) \notag \\
 & = \sum_{j\in\mathbb Z} \Delta_j ( \Delta_{\le j+3} f g) + \sum_{j\in \mathbb Z} \Delta_j ( \Delta_{>j+3} f \Delta_{\ge j+1} g) \notag\\
& =: G_1(f,g) +G_2(f,g). \label{G1G2}
\end{align}

We shall estimate $G_1$ and $G_2$ separately using the following lemma.

\begin{lem}[Paraproduct estimates] \label{lem_G1G2}
 Let $s>0$, $\sigma>0$, $1<p_i <\infty$, $i=1,\cdots, 6$. Then
\begin{align}
 \|G_1(f,g) \|_{\dot B^{-s}_{p,2}} &\lesssim \|f\|_{\dot B^{-s}_{p_1,2}} \cdot \|g\|_{p_2},
 \quad \frac 1p=\frac 1 {p_1}+\frac 1{p_2},\label{para_1}\\
\| G_2(f,g) \|_{\dot B^{-s}_{p,2}} & \lesssim \| f\|_{\dot B^{-s}_{p_3,2}} \cdot \|g\|_{\dot B^{s_1}_{p_4,\infty}},
\quad s_1>s, \, \frac 1{p_3}+ \frac 1 {p_4} = \frac 1 p+ \frac {s_1}{d}, \label{para_2}\\
\| G_2(f,g) \|_{\dot B^{\sigma}_{p,2}} & \lesssim \| f\|_{\dot
B^{-s}_{p_5,2}} \cdot \| g\|_{\dot B^{s+\sigma}_{p_6,\infty}}, \quad
\frac 1 {p_5} + \frac 1 {p_6} = \frac 1 p.\label{para_3}
\end{align}

\end{lem}
\begin{proof} In the proof of \eqref{para_1}, we use H\"older and
Lemma \ref{lem_Besov_def}. We have
\begin{align*}
\|G_1(f,g)\|_{\dot B^{-s}_{p,2}}&\lsm \biggl(\sum_{j\in \mathbb
Z}\biggl( 2^{-js}\|\triangle_{\le j+3} f
g\|_{p}\biggr)^2\biggr)^{\frac 12}\\
&\lsm \biggl(\sum_{j\in \mathbb Z}\biggl(2^{-js}\|\triangle_{\le
j+3} f\|_{p_1}\biggr)^2\biggr)^{\frac 12}\|g\|_{p_2}\\
&\lsm \|f\|_{\dot B^{-s}_{p_1,2}}\|g\|_{p_2}.
\end{align*}
We now prove \eqref{para_2}. From the definition and the Bernstein estimate
we have
\begin{align*}
\|G_2(f,g)\|_{\dot B^{-s}_{p,2}}&\lsm \biggl(\sum_{j\in\mathbb
Z}\biggl(2^{-js}\|\triangle_j(\triangle_{>j+3} f\triangle_{>j+1}
g)\|_p\biggr)^2\biggr)^{\frac 12}\\
&\lsm \biggl(\sum_{j\in \mathbb Z}\biggl(
2^{-js}2^{js_1}\|\triangle_j (\triangle_{>j+3} f\triangle_{>j+1}
g)\|_{\frac{pd}{d+ps_1}}\biggr)^2\biggr)^{\frac 12}\\
&\lsm \biggl(\sum_{j\in \mathbb
Z}\biggl(2^{j(s_1-s)}\sum_{k>j+3,|k-k'|\le 2}\|\triangle_k
f\|_{p_3}\|\triangle_{k'} g\|_{p_4}\biggr)^2\biggr)^{\frac 12}\\
&\lsm \biggl(\sum_{j\in \mathbb Z}\biggl(\sum_{k>j+3,|k-k'|\le 2}
2^{(j-k)(s_1-s)}2^{-ks}\|\triangle_k f\|_{p_3}
2^{ks_1}\|\triangle_{k'} g\|_{p_4}\biggr)^2\biggr)^{\frac 12}\\
&\lsm \biggl(\sum_{j\in \mathbb Z}\biggl(\sum_{k>j+3}
2^{(j-k)(s_1-s)}2^{-ks}\|\triangle_k
f\|_{p_3}\biggr)^2\biggr)^{\frac 12}\| g\|_{ \dot{B}^{s_1}_{p_4, \infty} }\\
&\lsm \|f\|_{\dot B^{-s}_{p_3,2}}\|g\|_{\dot B^{s_{1}}_{p_4,\infty}}.
\end{align*}
Here in the last line we have used the Young's inequality and the
fact that $s_1>s$.

Next we estimate \eqref{para_3}. We have
\begin{align*}
\|G_2(f,g)\|_{\dot B^{\sigma}_{p,2}}&\lsm \biggl(\sum_{j\in \mathbb
Z}\biggl(2^{j\sigma}\|\triangle_j(\triangle_{>j+3} f\triangle_{>j+1}
g)\|_p\biggr)^2\biggr)^{\frac 12}\\
&\lsm \biggl(\sum_{j\in \mathbb Z}\biggl (\sum_{k>j+3,|k-k'|\le
2}2^{j\sigma}\|\triangle_k f\|_{p_5}\|\triangle_{k'}
g\|_{p_6}\biggr)^2\biggr)^{\frac 12}\\
&\lsm \biggl(\sum_{j\in \mathbb Z}\biggl(\sum_{k>j+3,|k-k'|\le 2}
2^{(j-k)\sigma} 2^{-ks}\|\triangle_k f\|_{p_5}
2^{k(s+\sigma)}\|\triangle_{k'} g\|_{p_6}\biggl)^2\biggl)^{\frac
12}\\
&\lsm \|f\|_{\dot B^{-s}_{p_5,2}}\|g\|_{\dot
B^{s+\sigma}_{p_6,\infty}}.
\end{align*}

\end{proof}

We are now ready to turn our attention to the proof of Theorem $\ref{thm_unconditional}$.

\begin{proof}[Proof of Theorem \ref{thm_unconditional}]
 By the existence component of the local well-posedness result, we can construct a strong solution in $\dot S^1$. Therefore, without loss of generality, we assume $u$ is a strong solution
and satisfies
\begin{align*}
 \| u \|_{\dot S^1(I)} \lesssim C(\|u_0\|_{\dot H^1}, \| u_1\|_{L^2}).
\end{align*}
As before, we may also assume $t_0=0$. Now let $\delta = u -v$.  Clearly,
$\delta$ satisfies the equation
\begin{align*}
 \delta_{tt} -\Delta \delta = F(u) -F(v), \quad \ \delta(0)=0, \, \delta_t (0) =0.
\end{align*}
By the fundamental theorem of calculus, we write
\begin{align*}
F(u) - F(v) & = \delta \int_0^1 F^\prime (\lambda u + (1-\lambda) v) d\lambda \\
 & =\delta \int_0^1 ( F^\prime ((1-\lambda) \delta -u) - F^\prime (u) ) d\lambda + \delta F^\prime(u)\\
& = \delta \cdot H + \delta\cdot  F^\prime (u),
\end{align*}
where in the second equality we have used the fact that $F^\prime$ is an even function. Also due to the H\"older continuity of $F'(z)$, the function $H$ has the pointwise bound
\begin{align}\label{h_ptw}
|H(x)|\lsm |\delta(x)|^{\frac 4{d-2}}.
\end{align}

Let $I_0$ be a small time interval containing $0$. We shall choose $I_0$ sufficiently small later.
Using the Strichartz inequality and the Duhamel formula, we estimate
\begin{align}
 \| \delta \|_{L_t^2 \dot B^{-\frac 1{d-1}}_{\frac{2(d-1)}{d-3}, 2}(I_0 \times \R^d)}
& \lesssim \| G_1(\delta, F^\prime (u)) \|_{L_t^{\frac{2(d+1)}{d+5}}
\dot B^{-\frac 1{d-1}}
_{\frac{2(d^2-1)}{d^2+2d-7},2}(I_0\times \R^d)} \label{term1}\\
&\qquad+ \|G_2(\delta, F'(u))\|_{L_t^{\frac{2(d+1)}{d+5}} \dot
B^{-\frac 1{d-1}} _{\frac{2(d^2-1)}{d^2+2d-7},2}(I_0\times \R^d)}\label{term2}\\
 & \qquad + \| G_1(\delta, H) \|_{L_t^2 \dot B^{\frac 1{d-1}}_{\frac{2(d-1)}{d+1},2}
(I_0\times \R^d)}\label{term3}\\
& \qquad + \| G_2(\delta, H) \|_{L_t^2 \dot B^{\frac 1{d-1}}_{\frac{2(d-1)}{d+1},2}
 (I_0 \times \R^d)},\label{term4}
\end{align}
where $G_1(\cdot,\cdot)$, $G_2(\cdot,\cdot)$ are defined in \eqref{G1G2}.

To estimate \eqref{term1}, we use \eqref{para_1} with
\begin{align*}
s=\frac 1{d-1},\ p=\frac{2(d^2-1)}{d^2+2d-7},\
p_1=\frac{2(d-1)}{d-3},\ p_2=\frac{d+1} 2
\end{align*}
and H\"older in time to get
\begin{align*}
\eqref{term1}&\lsm \|\delta\|_{L_t^2\dot B^{-\frac
1{d-1}}_{\frac{2(d-1)}{d-3},2}(I_0\times\R^d)}\|F'(u)\|_{L_{t,x}^{\frac{d+1}2}(I_0\times\R^d)}\\
&\lsm \|\delta\|_{L_t^2\dot B^{-\frac
1{d-1}}_{\frac{2(d-1)}{d-3},2}(I_0\times\R^d)}
\|u\|_{L_{t,x}^{\frac{2(d+1)}{d-2}}(I_0\times\R^d)}^{\frac 4{d-2}}\\
&\lsm \|\delta\|_{L_t^2\dot B^{-\frac
1{d-1}}_{\frac{2(d-1)}{d-3},2}(I_0\times\R^d)}
\|u\|_{L_t^{\frac{2(d+1)}{d-2}}\dot B^{\frac
d{2(d-1)}}_{\frac{2(d^2-1)}{d^2-2d+3},2}(I_0\times\R^d)}^{\frac
4{d-2}}
\end{align*}
To estimate \eqref{term2}, we use \eqref{para_2} with the same $s,p$
and
\begin{align*}
s_1=\frac 2{d-1},\ p_3=\frac{2(d-1)}{d-3}, \
p_4=\frac{d(d^2-1)}{2(d^2+1)}
\end{align*}
in the space variable. In the time variable, we use the H\"older inequality.  We also use Lemma \ref{simplefact} with $f=F'$.
This gives us
\begin{align*}
\eqref{term2}&\lsm \|\delta\|_{L_t^2\dot B^{-\frac
1{d-1}}_{\frac{2(d-1)}{d-3},2}(I_0\times\R^d)}\|F'(u)\|_{L_t^{\frac{d+1}2}\dot
B^{\frac 2{d-1}}_{\frac{d(d^2-1)}{2(d^2+1)},\infty}(I_0\times\R^d)}\\
&\lsm \|\delta\|_{L_t^2\dot B^{-\frac
1{d-1}}_{\frac{2(d-1)}{d-3},2}(I_0\times\R^d)}\|u\|_{L_t^{\frac{2(d+1)}{d-2}}\dot
B^{\frac{d-2}{2(d-1)}}_{\frac{2d(d^2-1)}{(d-2)(d^2+1)},\infty}(I_0\times\R^d)}^{\frac
4{d-2}}\\
&\lsm \|\delta\|_{L_t^2\dot B^{-\frac
1{d-1}}_{\frac{2(d-1)}{d-3},2}(I_0\times\R^d)}\|u\|_{L_t^{\frac{2(d+1)}{d-2}}\dot
B^{\frac
d{2(d-1)}}_{\frac{2(d^2-1)}{d^2-2d+3},2}(I_0\times\R^d)}^{\frac
4{d-2}}
\end{align*}

To estimate \eqref{term4}, we use \eqref{para_3} with $\sigma=\frac
1{d-1}$, $p=\frac{2(d-1)}{d+1}$, $p_5=\frac{2(d-1)}{d-3}$,
$p_6=\frac{d-1}2$ and H\"older in time to get
\footnote{ A key point of
the following estimate is to separate a portion of ``$\delta$" when estimating $H$. Since $H$ is only bounded pointwise
by $|\delta|^{\frac 4{d-2}}$, instead of estimating $\|H\|_{\dot B^{\frac 2{d-1}}_{\frac{d-1}2,\infty}}$ directly,
we have to use the interpolation inequality to extract a portion of $L_{x}^{\frac d2}$ norm of $H$ which in turn can be bounded by
$L_{x}^{\frac{2d}{d-2}}$-norm of $\delta$. We thank F. Planchon for the correction on the previous text regarding this point.}

\begin{align}
\nonumber \eqref{term4}&\lsm \|\delta\|_{L_t^2 \dot B^{-\frac
1{d-1}}_{\frac{2(d-1)}{d-3},
2}(I_0\times\R^d)}\|H\|_{L_t^{\infty}\dot B^{\frac 2{d-1}}_{\frac
{d-1}2,\infty}(I_0\times\R^d)}\\
\label{eq_1}&\lsm \|\delta\|_{L_t^2 \dot B^{-\frac 1{d-1}}_{\frac{2(d-1)}{d-3},
2}(I_0\times\R^d)} \| H\|^{\frac 12}_{L_t^\infty \dot B^0_{\frac
d2,\infty} (I_0\times \R^d)} \| H \|^{\frac 12}_{L_t^\infty \dot
B^{\frac 4{d-1}}_{\frac {d(d-1)} {2(d+1)},\infty} (I_0 \times \R^d)
} \\
\label{eq_2}&\lsm \|\delta\|_{L_t^2 \dot B^{-\frac 1{d-1}}_{\frac{2(d-1)}{d-3},
2}(I_0\times\R^d)} \lVert |\delta(x)|^\frac{4}{d-2}\rVert_{L_t^\infty L^\frac{d}{2}}^\frac{1}{2} \\
\nonumber &\hspace{0.4in} \cdot\left(\int_0^1 \lVert F'((1-\lambda)\delta -u)-F'(u)\rVert_{L_t^\infty \dot{B}^\frac{4}{d-1}_{\frac{d(d-1)}{2(d+1)},\infty}(I_0\times \mathbb{R}^d)}d\lambda\right)^\frac{1}{2}\\
\nonumber& \lsm \|\delta\|_{L_t^2 \dot B^{-\frac 1{d-1}}_{\frac{2(d-1)}{d-3},
2}(I_0\times\R^d)} \lVert \delta\rVert_{L_t^\infty \dot{H}^1_x}^\frac{2}{d-2}\left(\lVert \delta\rVert_{L_t^\infty \dot{H}^1_x}^\frac{2}{d-2}+\lVert u\rVert_{L_t^\infty \dot{H}^1_x}^\frac{2}{d-2}\right),
\end{align}
where to obtain \eqref{eq_1} we use interpolation, to obtain \eqref{eq_2} we use the pointwise bound \eqref{h_ptw} and the definition of $H$. In the last line we have used Lemma \ref{simplefact} and Sobolev embedding.


Finally we estimate \eqref{term3}. By frequency localization, we further decompose $G_1(\delta,H)$ as
\begin{align}
 G_1(\delta, H) & = \sum_{j\in \mathbb Z} \Delta_j ( \Delta_{\le j+3} \delta \Delta_{\ge j-3} H ) \label{term5} \\
& \quad + \sum_{j \in \mathbb Z} \Delta_j (\Delta_{j-2\le \cdot \le j+3} \delta \Delta_{<j-3} H). \label{term6}
\end{align}
A quick observation shows that \eqref{term5} can be estimated in a similar way as \eqref{term4}. Therefore we have
\begin{align*}
 \eqref{term5} \lsm \|\delta\|_{L_t^2 \dot B^{-\frac 1{d-1}}_{\frac{2(d-1)}{d-3},
2}(I_0\times\R^d)} \cdot \| \delta \|^{\frac 2{d-2}}_{L_t^\infty
\dot H_x^1 (I_0\times \R^d)} \cdot ( \| \delta \|^{\frac
2{d-2}}_{L_t^\infty \dot H_x^1 (I_0\times \R^d)} + \| u \|^{\frac
2{d-2}}_{L_t^\infty \dot H_x^1 (I_0\times \R^d)} ).
\end{align*}
Now we turn to estimating \eqref{term6}. To simplify notation, observe that $\Delta_{j-2\le \cdot \le j+3} \delta$
essentially behaves as $\Delta_j \delta$. Therefore in the estimate below we write $\Delta_j \delta$ in place
of $\Delta_{j-2\le \cdot \le j+3} \delta$. With this convention, we have
\begin{align}
 \eqref{term6} & \lsm \biggl \| \biggl(  \sum_{j\in \mathbb Z} ( 2^{\frac 1{d-1} j}
 \|\Delta_j \delta \|_{\frac {2(d^3-d^2)}{d^3-3d^2+6d-2}} \cdot \| \Delta_{<j-3} H \|_{\frac{d^2}{2d-1}}  )^2
\biggr)^{\frac 12} \biggr \|_{L_t^2(I_0)} \notag \\
& \lsm \biggl\| \|\delta\|_{\dot B^{\frac 1{d-1} + \frac 1d}_{\frac {2(d^3-d^2)}{d^3-3d^2+6d-2}, 2}} \cdot
\| H \|_{\dot B^{-\frac 1{d}}_{\frac {d^2} {2d-1}, \infty}} \biggr\|_{L_t^2(I_0)}. \label{eq_1030_a}
\end{align}
By embedding we have
\begin{align*}
 \| H \|_{\dot B^{-\frac 1{d}}_{\frac {d^2} {2d-1}, \infty}} &\lsm \| H \|_{L_x^{\frac d2}}
\lsm \| \delta^{\frac 4{d-2}} \|_{L_x^{\frac d2}} \\
& \lsm \| \delta \|^{\frac 4{d-2}}_{\dot B^{\frac 1{d-1} + \frac 1d}_{\frac {2(d^3-d^2)}{d^3-3d^2+6d-2}, 2}}.
\end{align*}
By interpolation we have
\begin{align*}
 \|\delta\|_{\dot B^{\frac 1{d-1} + \frac 1d}_{\frac {2(d^3-d^2)}{d^3-3d^2+6d-2}, 2}}
\lsm \| \delta \|^{1+\frac 1{d^2} -\frac 3d}_{\dot B^{-\frac 1{d-1}}_{\frac{2(d-1)}{d-3},2}} \cdot
\| \delta\|_{\dot H^1}^{\frac 3d - \frac 1{d^2}}.
\end{align*}
Therefore
\begin{align*}
 \eqref{eq_1030_a} & \lsm \biggl\|
\|\delta \|_{\dot B^{\frac 1{d-1} + \frac 1d}_{\frac {2(d^3-d^2)}{d^3-3d^2+6d-2}, 2}}^{\frac{d+2}{d-2}}
\biggr\|_{L_t^2(I_0)} \\
& \lsm \biggl \| \| \delta \|^{(1+\frac 1{d^2} -\frac 3 d)
\cdot \frac{d+2}{d-2} }_{\dot B^{-\frac 1{d-1}}_{\frac{2(d-1)}{d-3},2}}
\cdot \| \delta \|_{\dot H^1}^{(\frac 3d-\frac 1{d^2})\cdot \frac{d+2}{d-2}} \biggr\|_{L_t^2(I_0)} \\
& \lsm \|
\delta \|_{L_t^2 \dot B^{-\frac 1{d-1}}_{\frac{2(d-1)}{d-3},2} (I_0\times \R^d)}
\cdot \| \delta\|^{\frac 4{d-2}}_{L_t^\infty \dot H^1(I_0 \times \R^d)}.
\end{align*}

Collecting all the estimates, we get
\begin{align*}
 & \|\delta \|_{L_t^2 \dot B^{-\frac 1{d-1}}_{\frac{2(d-1)}{d-3},2} (I_0 \times \R^d)} \\
\lsm &\|\delta \|_{L_t^2 \dot B^{-\frac
1{d-1}}_{\frac{2(d-1)}{d-3},2} (I_0 \times \R^d)} \cdot \biggl (
\|\delta \|_{L_t^\infty \dot H^1 (I_0\times \R^d)}^{\frac 4{d-2}}+
\|u\|_{L_t^{\frac{2(d+1)}{d-2}}\dot B^{\frac
d{2(d-1)}}_{\frac{2(d^2-1)}{d^2-2d+3},2}(I_0\times\R^d)}^{\frac
4{d-2}} \\
& \qquad+\| \delta \|^{\frac 2{d-2}}_{L_t^\infty \dot H_x^1
(I_0\times \R^d)}  \cdot \| u \|^{\frac 2{d-2}}_{L_t^\infty \dot
H_x^1 (I_0\times \R^d)}\biggr).
\end{align*}
Observing that $\delta \in C(I_0, \dot H^1)$, $\delta(0)=0$ and noting the boundedness
of
\begin{align*}
\|u\|_{L_t^{\frac{2(d+1)}{d-2}}\dot
B^{\frac{d}{2(d-1)}}_{\frac{2(d^2-1)}{d^2-2d+3},2}(I_0\times\R^d)}^{\frac{4}{d-2}},
\end{align*}
 we conclude that for $I_0$ sufficiently small, $\delta=0$ on $I_0$. A simple bootstrap argument then yields
that $\delta = 0$ on the whole interval $I$. The theorem is proved.
\end{proof}

\begin{rem} \label{rem1043}
Interestingly, the proof of unconditional uniqueness also provides a proof of local well-posedness in
high dimensions $d\ge 5$. We briefly sketch the argument as follows. Define the map
\begin{align*}
\phi(u) = K(t) (u_0,u_1) + \int_0^t \frac{\sin ((t-\tau) \sqrt{-\Delta})}{\sqrt{-\Delta}}F(u(\tau)) d\tau.
\end{align*}
Let $\delta>0$ (to be fixed later) and choose the time interval $I$ sufficiently small such that
\begin{align*}
\| K(t) (u_0,u_1) \|_{L_t^2 \dot B^{-\frac 1{d-1}}_{\frac{2(d-1)}{d-3},2} (I)}
+\| K(t) (u_0,u_1) \|_{L_t^{\frac{2(d+1)}{d-1}} \dot B^{\frac 12}_{\frac{2(d+1)}{d-1},2} (I)} \\
+\| K(t) (u_0,u_1) \|_{L_t^{\frac{2(d+1)}{d-2}} \dot B^{\frac {d}{2(d-1)}}_{\frac{2(d^2-1)}{d^2-2d+3},2} (I)}
\le \delta.
\end{align*}
Then consider the ball
\begin{align*}
B_1 = \biggl \{ u \in \dot S^1(I):\;
 \|u\|_{L_t^2 \dot B^{-\frac 1{d-1}}_{\frac{2(d-1)}{d-3},2}(I)} \le 2\delta,
\|u\|_{L_t^{\frac{2(d+1)}{d-2}} \dot B^{\frac{d}{2(d-1)}}_{\frac{2(d^2-1)}{d^2-2d+3},2} (I)} \le 2\delta, \\
\, \text{and}\quad
\|u\|_{L_t^{\frac{2(d+1)}{d-1}} \dot B^{\frac 12}_{\frac{2(d+1)}{d-1},2} (I)} \le 2 \delta \biggr\}.
\end{align*}
It is not difficult to check that $\phi$ maps $B_1$ into $B_1$ for $\delta$ sufficiently small.
Furthermore by using estimates similar to \eqref{term1}, \eqref{term2},
we have
\begin{align*}
&\| \phi(u) -\phi(v) \|_{L_t^2\dot B^{-\frac 1{d-1}}_{\frac{2(d-1)}{d-3},2} (I)} \\
\lsm & \| u-v\|_{L_t^2\dot B^{-\frac 1{d-1}}_{\frac{2(d-1)}{d-3},2} (I)}
\cdot ( \|u \|^{\frac 4{d-2}}_{L_t^{\frac{2(d+1)}{d-2}} \dot B^{\frac{d}{2(d-1)}}_{\frac{2(d^2-1)}{d^2-2d+3},2} (I)}
+\|v \|^{\frac 4{d-2}}_{L_t^{\frac{2(d+1)}{d-2}} \dot B^{\frac{d}{2(d-1)}}_{\frac{2(d^2-1)}{d^2-2d+3},2} (I)}) \\
\lsm & \delta^{\frac 4{d-2}} \cdot \| u-v\|_{L_t^2\dot B^{-\frac 1{d-1}}_{\frac{2(d-1)}{d-3},2} (I)},
\end{align*}
for all $u$, $v\in B_1$. This shows that $\phi$ is a contraction on $B_1$ if $\delta$ is sufficiently small and therefore we can find a unique solution
in $B_1$.
\end{rem}

We conclude this section by giving the proof of Lemma $\ref{blowup_criterion_lemma}$.
\begin{proof}[Proof of Lemma $\ref{blowup_criterion_lemma}$]
Denote $L=\|u\|_{S([t_0,T_0])}$. We divide the proof into two steps.

Step 1. We show that
\begin{align}
\|u\|_{\dot S^1([t_0,T_0])} \le A:=C(L,d) \cdot (\lVert u_0\rVert_{\dot{H}^1}+\lVert u_1\rVert_{L^2}). \label{eq956}
\end{align}

 let $\xi>0$ be given (to be fixed later in the argument).  First, we partition $[t_0,T_0]$ into $N=N(L,\xi,d)$ intervals $I_j=[t_j,t_{j+1}]$ such that
\begin{align*}
\lVert u\rVert_{S(I_j)}\leq \xi.
\end{align*}
Then by the Strichartz inequality, we get
\begin{align*}
\lVert u\rVert_{\dot{S}^1(I_j)}&\lesssim \lVert u(t_j)\rVert_{\dot{H}^1}+\lVert \partial_t u(t_j)\rVert_{L^2}+\lVert u\rVert_{S(I_j)}^\frac{4}{d-2}\lVert u\rVert_{\dot{S}^1(I_j)}\\
&\lesssim \lVert u(t_j)\rVert_{\dot{H}^1}+\lVert \partial_t u(t_j)\rVert_{L^2}+\xi^\frac{4}{d-2}\lVert u\rVert_{\dot{S}^1(I_j)}
\end{align*}
 Thus,
\begin{align*}
\lVert u\rVert_{\dot{S}^1(I_j)}\lesssim \lVert u(t_j)\rVert_{\dot{H}^1}+\lVert \partial_t u(t_j)\rVert_{L^2}
\end{align*}
for $\xi$ sufficiently small. A simple induction then shows that
\begin{align*}
\lVert u\rVert_{\dot{S}^1([t_0,T_0])}\leq C(L,d)\cdot (\lVert u_0\rVert_{\dot{H}^1}+\lVert u_1\rVert_{L^2}).
\end{align*}

Step 2. By the local well-posedness Theorem \ref{lwpthm}, it is enough to show the existence of $\epsilon$ and $\delta$
such that
\begin{align}
\| K(t-(T_0-\epsilon)) (u(T_0-\epsilon), (\partial_t u)(T_0-\epsilon)) \|_{S(
 [T_0-\epsilon, T_0+\delta]) }
\le \eta, \label{eq1055p}
\end{align}
where $\eta=\eta(A)$ is sufficiently small (specified by Theorem \ref{lwpthm}). We first estimate the piece on $[T_0-\epsilon, T_0]$, i.e.
\begin{align}
\| K(t-(T_0-\epsilon)) (u(T_0-\epsilon), (\partial_t u)(T_0-\epsilon)) \|_{S(
 [T_0-\epsilon, T_0])} \label{eq938}
\end{align}
Using Duhamel, Strichartz and \eqref{eq956}, we get
\begin{align*}
\eqref{eq938} &\lesssim \| u\|_{S([T_0-\epsilon,T_0])} +\| |\nabla|^{\frac 12} F(u) \|_{L_{t,x}^{\frac{2(d+1)}{d+3}}([T_0-\epsilon,T_0])} \\
&\lesssim \|u\|_{S([T_0-\epsilon, T_0])} + \|u\|_{\dot S^1([T_0-\epsilon, T_0])}
\cdot \|u\|^{\frac 4{d-2}}_{S([T_0-\epsilon,T_0])} \\
&\lesssim \|u\|_{S([T_0-\epsilon, T_0])} + A \cdot \|u\|^{\frac 4{d-2}}_{S([T_0-\epsilon,T_0])}. \\
\end{align*}
Clearly we can choose $\epsilon$ sufficiently small to obtain
\begin{align*}
\eqref{eq938} \le \frac {\eta}2.
\end{align*}
Now since $\epsilon$ is fixed, by Lebesgue monotone convergence, there exists $\delta$ sufficiently small, such that
\begin{align*}
\| K(t-(T_0-\epsilon)) (u(T_0-\epsilon), (\partial_t u)(T_0-\epsilon)) \|_{S(
 [T_0, T_0+\delta])} \le \frac {\eta}2.
\end{align*}
Therefore by adding the two pieces together we have proved \eqref{eq1055p}. By Theorem $\ref{lwpthm}$, it follows
that there exists a unique solution $v$ to (NLW) on $[T_0-\epsilon,T_0+\delta]$ with $v(T_0-\epsilon)=u(T_0-\epsilon)$.
We then use unconditional uniqueness, Theorem $\ref{thm_unconditional}$, to see that $u=v$ on $[t_0,T_0]$ and thus $v$ gives the desired extension.
\end{proof}

\section{Long time perturbation} \label{sec-perturbation}

In this section, we prove a long-time perturbation result for (NLW).
We start with the following short-time perturbation theorem.

\begin{thm}[Short time perturbation] \label{thm_short}
Let $(u_0, u_1) \in \dot H^1 \times L^2$. Let $\tilde u$ be a near solution in the following sense
\begin{align*}
 \partial_{tt} \tilde u -\Delta \tilde u= F(\tilde u) +e
\end{align*}
such that
\begin{enumerate}
\item[(a)]
\begin{align*}
     \|\tilde u_0 -u_0 \|_{\dot H^1}+ \| \tilde u_1 -u_1\|_{L^2} \le A^\prime
    \end{align*}

\item[(b)] Smallness:
\begin{align*}
 \| K(t) ( u_0 -\tilde u_0, u_1-\tilde u_1 ) \|_{X(I) } \le \epsilon,
\end{align*}
\begin{align*}
 \| |\nabla|^{\frac 12} e \|_{L_{t,x}^{\frac{2(d+1)}{d+3}}(I)} \le \epsilon,
\end{align*}
\begin{align*}
 \|\tilde u\|_{X(I)} +\| \tilde u \|_{W(I)} + \| \tilde u \|_{S(I)}
 + \| \tilde u \|_{Y(I)}\le \delta.
\end{align*}
\end{enumerate}
Then for $0<\delta \le \delta_0(d)$, $0<\epsilon \le \epsilon_0(A^\prime)$,
(NLW) with initial data $(u_0,u_1)$, there exists unique
$u$ on $I\times \R^d$ such that
\begin{align}
\|u-\tilde u \|_{\dot S^1} & \le C(d) \cdot A^\prime, \label{e10a}\\
\| u-\tilde u \|_{X(I)} &\le C(d,A^\prime) \cdot \epsilon, \label{e10b} \\
\|F(u) -F(\tilde u) \|_{X^\prime(I)} & \le C(d,A^\prime) \cdot \epsilon. \label{e10c}
\end{align}
\end{thm}

\begin{proof}
Assume $u$ exists on $I$. Then
\begin{align*}
 \partial_{tt} (\tilde u - u) -\Delta (\tilde u - u)= F(\tilde u) - F(u) +e.
\end{align*}
Now
\begin{align}
 \| \tilde u - u \|_{X(I)} & \lesssim  \| K(t) (u_0-\tilde u_0, u_1 -\tilde u_1) \|_{X(I)}  \label{e1} \\
& \qquad+ \left\| \int_0^t \frac{ \sin ((t-\tau)\sqrt{-\Delta})} {\sqrt{-\Delta}}
e(\tau) d\tau \right\|_{X(I)} \label{e2} \\
& \qquad + \left \| \int_0^t  \frac{ \sin ((t-\tau)\sqrt{-\Delta})} {\sqrt{-\Delta}}
(F(\tilde u) - F(u)) d\tau \right\|_{X(I)} \label{e3}.
\end{align}
The estimate of \eqref{e1} follows from the assumption and we have
\begin{align}
 \eqref{e1} \lesssim \epsilon.  \label{etmp_1}
\end{align}
To estimate \eqref{e2}, we use Lemma \ref{lem3} and Strichartz,
\begin{align}
 \eqref{e2} & \lesssim \left\| \int_0^t  \frac{ \sin ((t-\tau)\sqrt{-\Delta})} {\sqrt{-\Delta}}
e(\tau) d\tau \right\|_{\dot S^1(I) } \notag\\
& \lesssim \left\| |\nabla|^{\frac 12} e \right\|_{L_{t,x}^{\frac {2(d+1)}{d+3}} (I)} \notag\\
& \lesssim \epsilon. \label{etmp_2}
\end{align}
To estimate \eqref{e3}, we use Lemma \ref{lem4}, Lemma \ref{lem4a} and Lemma \ref{lem3} to get
\begin{align}
 \eqref{e3} & \lesssim \| F(u) -F(\tilde u) \|_{X^\prime(I)} \notag\\
& \lesssim
\| u-\tilde u\|_{X(I)} \cdot ( \|u-\tilde u\|^{\frac 4{d-2}}_{S(I)} + \|\tilde u\|^{\frac 4{d-2}}_{S(I)} )
 \notag \\
& \quad+ \|u-\tilde u\|_{X(I)} \cdot ( \|u-\tilde u\|_{X(I)}^{\theta_4} \cdot \|u-\tilde u\|_{\dot S^1(I)}^{1-\theta_4} +
\|\tilde u\|_{X(I)}^{\theta_4} \cdot \| \tilde u\|_{Y(I)}^{1-\theta_4} )^{\frac 4d} \notag\\
& \qquad \cdot ( \|u-\tilde u\|_{S(I)} +  \|\tilde u\|_{S(I)} )^{\frac 8{d(d-2)}} \notag\\
& \lesssim
\| u-\tilde u\|_{X(I)} \cdot ( \|u-\tilde u\|^{\frac 4{d-2} \theta_2 }_{X(I)} \cdot
\|u-\tilde u\|^{\frac 4{d-2} (1-\theta_2) }_{W(I)}+ \delta^{\frac 4{d-2}} ) \notag\\
& \quad+ \|u-\tilde u\|_{X(I)} \cdot ( \|u-\tilde u\|_{X(I)}^{\theta_4} \cdot \|u-\tilde u\|_{\dot S^1(I)}^{1-\theta_4} +
\delta)^{\frac 4d} \notag\\
& \qquad \cdot ( \|u-\tilde u\|^{ \theta_2 }_{X(I)} \cdot
\|u-\tilde u\|^{ 1-\theta_2 }_{W(I)}+ \delta)^{\frac 8{d(d-2)}}.
\label{etmp_3}
\end{align}
Collecting the estimates \eqref{etmp_1}, \eqref{etmp_2} and \eqref{etmp_3} and using the
fact that $\dot S^1 \hookrightarrow W$, we obtain
\begin{align}
&\| \tilde u-  u\|_{X(I)} \notag\\
&\hspace{0.2in} \lesssim\;   \epsilon
 + \|u-\tilde u\|_{X(I)} \cdot (\|u-\tilde u\|_{X(I)}^{\frac 4{d-2} \theta_2} \cdot
     \|u-\tilde u\|^{\frac 4{d-2}(1-\theta_2)}_{\dot S^1(I)} + \delta^{\frac 4{d-2}} )
\notag\\
&\hspace{0.4in} + \|u-\tilde u\|_{X(I)} \cdot ( \|u-\tilde u\|_{X(I)}^{\theta_4} \cdot \|u-\tilde u\|_{\dot S^1(I)}^{1-\theta_4} +
\delta)^{\frac 4d} \notag\\
&\hspace{0.4in}\cdot ( \|u-\tilde u\|^{ \theta_2 }_{X(I)} \cdot
\|u-\tilde u\|^{ 1-\theta_2 }_{\dot S^1(I)}+ \delta)^{\frac 8{d(d-2)}}.
\label{e104_a}
\end{align}
This is the first estimate we need. Next we estimate $\|\tilde u - u \|_{\dot S^1(I)}$. By the Strichartz inequality and
Lemma \ref{lem4a}, we have
\begin{align}
&\|\tilde u -u \|_{\dot S^1(I)} \notag \\
&\hspace{0.2in} \lesssim   \;A^\prime+\epsilon + \| F(\tilde u)-F(u) \|_{W^\prime(I)} \notag \\
&\hspace{0.2in}\lesssim  \; A^\prime + \epsilon+ \| \tilde u-u \|_{W(I)} \cdot ( \|\tilde u-u\|^{\frac 4{d-2}}_{S(I)} +
\|\tilde u\|^{\frac 4{d-2}}_{S(I)} )
 + \|\tilde u -u\|^{\frac 4{d-2}}_{S(I)} \cdot \|\tilde u\|_{W(I)} \notag \\
&\hspace{0.2in} \lesssim   \; A^\prime + \epsilon+ \|\tilde u-u\|_{\dot S^1(I)} \cdot
 ( \|\tilde u -u\|_{X(I)}^{\frac 4{d-2}\theta_2} \|\tilde u-u\|_{\dot S^1(I)}^{\frac 4{d-2}(1-\theta_2)}
 +\delta^{\frac 4{d-2}} ) \notag \\
& \hspace{0.4in} + \|\tilde u -u\|_{X(I)}^{\frac 4{d-2} \theta_2 } \cdot
\|\tilde u -u\|_{\dot S^1(I)}^{\frac 4{d-2} (1-\theta_2)} \cdot \delta. \label{e104_b}
\end{align}
Now by \eqref{e104_a}, \eqref{e104_b} and a continuity argument, we get \eqref{e10a}, \eqref{e10b}
for sufficiently small $\epsilon \le \epsilon_0(A^\prime) $ and $\delta \le \delta_0(d)$. We stress here
that $\delta$ can be chosen to depend only on the dimension $d$. Plugging the estimates \eqref{e10a}, \eqref{e10b}
into \eqref{etmp_3}, we also obtain \eqref{e10c}. The theorem is proved.
\end{proof}

Theorem \ref{thm_short} treats the case when $\epsilon \ll A^\prime$. In such a case all the constants in
\eqref{e10b}--\eqref{e10c} depend on $A^\prime$. One may wonder what happens when $A^\prime$ is of the same
order as $\epsilon$. In that case, similar arguments as in the proof of Theorem \ref{thm_short} give
the following corollary.
\begin{cor}[Short time perturbation, $\epsilon$-perturbation version] \label{cor_short}
 Let $(u_0, u_1) \in \dot H^1 \times L^2$. Let $\tilde u$ be a near solution in the following sense
\begin{align*}
 \partial_{tt} \tilde u -\Delta \tilde u= F(\tilde u) +e
\end{align*}
such that
\begin{enumerate}
\item[(a)]
 \begin{align*}
     \|\tilde u_0 -u_0 \|_{\dot H^1}+ \| \tilde u_1 -u_1\|_{L^2} \le \epsilon
    \end{align*}

\item[(b)] Smallness:
\begin{align*}
 \| |\nabla|^{\frac 12} e \|_{L_{t,x}^{\frac{2(d+1)}{d+3}}(I)} \le \epsilon,
\end{align*}
\begin{align*}
 \|\tilde u\|_{X(I)} +\| \tilde u \|_{W(I)} + \| \tilde u \|_{S(I)}
 + \| \tilde u \|_{Y(I)}\le \delta.
\end{align*}
\end{enumerate}
Then for $0<\delta \le \delta_0(d)$, $0<\epsilon \le \epsilon_0(d)$,
(NLW) with initial data $(u_0,u_1)$, there exists unique
$u$ on $I\times \R^d$ such that
\begin{align}
\|u-\tilde u \|_{\dot S^1} & \le C(d) \cdot \epsilon^{c}, \label{e10aa}\\
\| u-\tilde u \|_{X(I)} &\le C(d) \cdot \epsilon, \label{e10ab} \\
\|F(u) -F(\tilde u) \|_{X^\prime(I)} & \le C(d) \cdot \epsilon. \label{e10ac}
\end{align}
Here $0<c<1$ is a constant depending only on the dimension $d$.
\end{cor}
\begin{proof}
One only needs to repeat the derivation of \eqref{e104_a} and \eqref{e104_b} as in
the proof of Theorem \ref{thm_short}. We omit the details.
\end{proof}

Next we establish the long time perturbation in Besov spaces by using the short time perturbation result, Theorem
\ref{thm_short}.

\begin{thm}[Long time perturbation, Besov version] \label{thm_long1}
Assume $\tilde u$ is a near solution on  $I\times \R^d$
\begin{align*}
 \partial_{tt} \tilde u -\Delta\tilde u= F(\tilde u) +e
\end{align*}
such that
\begin{enumerate}
\item[(a)]
 \begin{align*}
    \| \tilde u\|_{\dot S^1(I)} \le E.
   \end{align*}

\item[(b)]
\begin{align*}
    \|\tilde u_0 -u_0 \|_{\dot H^1} + \| \tilde u_1 -u_1 \|_{L^2} \le E^\prime.
   \end{align*}

\item[(c)]
Smallness:
\begin{align*}
 \| K(t) (\tilde u_0 - u_0, \tilde u_1 -u_1) \|_{X(I)} & \le \epsilon  \\
 \| |\nabla|^{\frac 12} e \|_{L_{t,x}^{\frac{2(d+1)}{d+3}}(I)} \le \epsilon.
\end{align*}
\end{enumerate}

Then there exists $\epsilon_0 = \epsilon_0(d,E^\prime,E)$ such that if
$0<\epsilon<\epsilon_0$, for (NLW) with initial data $(u_0,u_1)$, there exists
a unique solution $u$ on $I\times \R^d$ with the properties
\begin{align*}
 \| \tilde u - u\|_{\dot S^1(I)} & \le C(d,E^\prime,E) \cdot E^\prime, \\
\|\tilde  u- u\|_{X(I)} & \le C(d,E^\prime,E) \cdot \epsilon.
\end{align*}

\end{thm}

\begin{proof}
Let $\delta_0=\delta_0(d)$ be chosen in the way as in Theorem \ref{thm_short}.
Denote $t_0=0$. Partition the time interval $I$ into $I=\bigcup_{j=1}^k I_j=\bigcup_{j=1}^k [t_{j-1},t_j]$
such that on each subinterval $I_j$
\begin{align*}
 \|\tilde u\|_{S(I_j)} + \| \tilde u\|_{W(I_j)} + \| \tilde u\|_{X(I_j)}
+\|\tilde u\|_{Y(I_j)} < \delta_0.
\end{align*}
One can choose $k=O( ( E /{\delta_0} )^{C(d)} )$ such intervals, where
$C(d)$ is a constant depending only on the dimension $d$. Consider the first
subinterval $I_1=[0,t_1]$. By Theorem \ref{thm_short}, for $\epsilon$ sufficiently small
depending only on $(d,E^\prime)$, we have
\begin{align*}
 \| \tilde u -u \|_{\dot S^1(I_1)} & \le C(d,E^\prime)\cdot E^\prime,\\
\| \tilde u - u \|_{X(I_1)} & \le C(d,E^\prime) \cdot \epsilon,\\
\| F(\tilde u) -F(u) \|_{X^\prime (I_1)} & \le C(d,E^\prime) \cdot \epsilon.
\end{align*}
Next for $1\le i\le k-1$, make the inductive assumption that
\begin{align}
 \|\tilde u - u\|_{\dot S^1(I_i)} & \le C_i(d,E^\prime, E) \cdot E^\prime, \label{e21a} \\
\| \tilde u - u\|_{X(I_i)} & \le C_i(d,E^\prime,E) \cdot \epsilon, \label{e21b} \\
\|F(\tilde u) -F(u) \|_{X^\prime ([0,t_i])} & \le C_i(d,E^\prime,E) \cdot \epsilon. \label{e21c}
\end{align}
Then for $I_{i+1}=[t_i,t_{i+1}]$ we will apply Theorem \ref{thm_short} with a time
shift $t_i$. For this we have to check the hypotheses of Theorem \ref{thm_short}.
To this end, by \eqref{e21a}, we have
\begin{align*}
 \|\tilde u(t_i) -u(t_i)\|_{\dot H^1}
+\|(\partial_t \tilde u)(t_i) -(\partial_t u)(t_i) \|_{L_x^2}
\le 2 C_i(d,E^\prime, E) \cdot E^\prime.
\end{align*}
Next by using Duhamel's formula and \eqref{e21c}, we have
\begin{align*}
 & \| K(t-t_i) ( \tilde u(t_i)-u(t_i), (\partial_t \tilde u)(t_i) -(\partial_t u)(t_i) ) \|_{X(I_{i+1})} \\
&\hspace{0.2in}\lesssim   \| K(t) (\tilde u_0 -u_0, \tilde u_1 -u_1) \|_{X(I)}
+ \|F(\tilde u) -F(u) \|_{X^\prime ([0,t_i])} + \||\nabla|^{\frac 12} e \|_{L_{t,x}^{\frac{2(d+1)}{d+3}}(I)} \\
&\hspace{0.2in}\lesssim   (1+ c_i(d,E^\prime,E)) \cdot \epsilon.
\end{align*}
Then for $\epsilon$ sufficiently small depending only on $(d,E^\prime,E)$, Theorem \ref{thm_short}
and \eqref{e21c} give
\begin{align*}
 \|\tilde u-u\|_{\dot S^1(I_{i+1})} & \le C_{i+1}(d,E^\prime,E) \cdot E^\prime, \\
\|\tilde u-u \|_{X(I_{i+1})} & \le C_{i+1}(d,E^\prime, E) \cdot \epsilon, \\
\|F(\tilde u) -F(u)\|_{X^\prime ([0,t_{i+1}])} &\le
\|F(\tilde u)-F(u)\|_{X^\prime ([0,t_i])} + \|F(\tilde u) -F(u)\|_{X^\prime ([t_i,t_{i+1}])} \\
& \le C_{i+1}(d,E^\prime,E) \cdot \epsilon.
\end{align*}
Consequently we have verified \eqref{e21a}--\eqref{e21c} for all $1\le i\le k$.
We stress that the choice of $\epsilon$ is consistent
since $k=O ( ( E{/\delta_0})^{C(d)} )$
is finite and we only need to adjust $\epsilon$ at most $k$ times. The
theorem now follows by summing \eqref{e21a}--\eqref{e21b}.
\end{proof}

Similar to the derivation of Corollary \ref{cor_short}, the same arguments as in the proof
of Theorem \ref{thm_long1} give the following result.
\begin{cor}[Long time perturbation, Besov $\epsilon$-perturbation version] \label{cor_long1a}
Assume $\tilde u$ is a near solution on  $I\times \R^d$
\begin{align*}
 \partial_{tt} \tilde u -\Delta \tilde u= F(\tilde u) +e
\end{align*}
such that
\begin{enumerate}
\item[(a)]
 \begin{align*}
    \| \tilde u\|_{\dot S^1(I)} \le E.
   \end{align*}

\item[(b)]
 \begin{align*}
    \|\tilde u_0 -u_0 \|_{\dot H^1} + \| \tilde u_1 -u_1 \|_{L^2} \le \epsilon.
   \end{align*}

\item[(c)]
Smallness:
\begin{align*}
 \| |\nabla|^{\frac 12} e \|_{L_{t,x}^{\frac{2(d+1)}{d+3}}(I)} \le \epsilon.
\end{align*}
\end{enumerate}

Then there exists $\epsilon_0 = \epsilon_0(d,E)$ such that if
$0<\epsilon<\epsilon_0$, for (NLW) with initial data $(u_0,u_1)$, there exists
a unique solution $u$ on $I\times \R^d$ with the properties
\begin{align}
 \| \tilde u - u\|_{\dot S^1(I)} & \le C(d,E) \cdot \epsilon^{c_1}, \label{e635a}\\
\|\tilde  u- u\|_{X(I)} & \le C(d,E) \cdot \epsilon. \notag
\end{align}
Here $0<c_1<1$ is a constant depending on $(d,E)$.
\end{cor}
\begin{proof}
This is essentially a repetition of the proof of Theorem \ref{thm_long1}. Note that
in \eqref{e635a} the constant $c_1$ depends both on the dimension $d$ and $E$. This
is a consequence of the short time perturbation theory (Corollary \ref{cor_short})
where we lose a power of $c$ due to the H\"older continuity of the nonlinearity. The additional dependence on $E$ comes
from the fact that we have to apply the short time theory $O(E^{C(d)})$ times.
\end{proof}

To obtain the usual Sobolev space version of Theorem \ref{thm_long1}, we need the following lemma, which shows that the $\dot{S}^1$ norm of the solution of the perturbed equation is bounded.

\begin{lem}[Boundedness of near solutions in Besov spaces] \label{lem7a}
Let $\tilde u$ be a near solution on $I\times \R^d$
\begin{align*}
 \partial_{tt} \tilde u -\Delta \tilde u= F(\tilde u) + e,
\end{align*}
such that
\begin{align}
 \|\tilde u \|_{L_t^\infty \dot H^1(I) } +
 \| \partial_t  \tilde u\|_{L_t^\infty L^2 (I)}
 &+\| \tilde u \|_{L_t^{\frac{2(d+1)}{d-1}} \dot H^{\frac 12,
\frac{2(d+1)}{d-1}} (I)} \le E, \label{e113a} \\
\| e\|_{L_t^{\frac{2(d+1)}{d+3}} \dot H^{\frac 12,\frac {2(d+1)}{d+3}} (I)} & \le E, \label{e113b}
\end{align}
Then
\begin{align} \label{eq_1231a}
 \| \tilde u \|_{\dot S^1 (I)} \le C(d,E).
\end{align}
\end{lem}
\begin{proof}
 We first have the interpolation inequality
\begin{align} \label{e112}
 \| \tilde u \|_{L_{t,x}^{\frac{2(d+1)}{d-2}}} \lesssim
\|\tilde u\|^{\frac 1{d-1}}_{L_t^\infty \dot H^1} \| \tilde
u\|^{\frac{d-2}{d-1}}_{L_t^{\frac{2(d+1)}{d-1}} \dot H^{\frac 12,
\frac{2(d+1)}{d-1}} }.
\end{align}
Now by \eqref{e113a}, \eqref{e113b}, Strichartz and \eqref{e112}, we
have
\begin{align*}
 \| \tilde u \|_{\dot S^1 (I)} &\lesssim E +
  \| \tilde u\|^{\frac 4{d-2}}_{L_{t,x}^{\frac{2(d+1)}{d-2}} (I)} \cdot \| \tilde u
\|_{L_t^{\frac{2(d+1)}{d-1}} \dot H^{\frac 12, \frac {2(d+1)}{d-1}} (I)}
\\
& \lesssim E+   E^{1+\frac 4 {d-2}} .
\end{align*}
This immediately  gives us \eqref{eq_1231a}.
\end{proof}

We are now ready to prove the main perturbation result stated
in Theorem \ref{thm_long2}.

\begin{proof}
By \eqref{e133a}, \eqref{e133c} and taking $\epsilon<E$, Lemma \ref{lem7a} gives us
\begin{align*}
 \| \tilde u\|_{\dot S^1(I)} \le C(d,E).
\end{align*}
By \eqref{e133b} and Lemma \ref{lem3}, we have
\begin{align*}
 & \| K(t) (\tilde u_0 - u_0, \tilde u_1 -u_1) \|_{X(I)}\\
&\hspace{0.2in} \lesssim \| K(t) (\tilde u_0 - u_0, \tilde u_1 -u_1) \|_{S(I)}^{\theta_1} \cdot
\| K(t) (\tilde u_0 - u_0, \tilde u_1 -u_1) \|_{L_t^\infty \dot H^1(I)}^{1-\theta_1}  \\
&\hspace{0.2in} \lesssim \epsilon^{\theta_1} \cdot (E^\prime)^{1-\theta_1}.
\end{align*}
Now for $\epsilon$ sufficiently small depending on $(d,E^\prime,E)$, we can apply
Theorem \ref{thm_long1} to obtain that
\begin{align}
 \| \tilde u - u\|_{\dot S^1(I)} & \le C(d,E^\prime,E) \cdot E^\prime, \label{e211a} \\
\|\tilde  u- u\|_{X(I)} & \le C(d,E^\prime,E) \cdot \epsilon. \label{e211b}
\end{align}
Finally \eqref{e210} follows from \eqref{e211b}, Lemma \ref{lem3} and \eqref{e211a}.
The theorem is proved.
\end{proof}

Finally we have the $\epsilon$-perturbation version of Theorem \ref{thm_long2} similar to
Corollary \ref{cor_long1a}. We omit the proof and leave the details to interested readers.
\begin{cor}[Long time perturbation, Sobolev $\epsilon$-perturbation version] \label{cor_long2}
Assume $\tilde u$ is a near solution on  $I\times \R^d$
\begin{align*}
 \partial_{tt} \tilde u-\Delta \tilde u = F(\tilde u) +e
\end{align*}
such that
\begin{enumerate}
\item[(a)]
\begin{align*}
   \|\tilde u \|_{L_t^\infty \dot H^1(I) } +
 \| \partial_t  \tilde u\|_{L_t^\infty L^2 (I)}
 +\| \tilde u \|_{L_t^{\frac{2(d+1)}{d-1}} \dot H^{\frac 12,
\frac{2(d+1)}{d-1}} (I)} \le E.
   \end{align*}

\item[(b)]
\begin{align*}
    \|\tilde u_0 -u_0 \|_{\dot H^1} + \| \tilde u_1 -u_1 \|_{L^2} \le \epsilon.
   \end{align*}

\item[(c)] Smallness:
\begin{align*}
 \| |\nabla|^{\frac 12} e \|_{L_{t,x}^{\frac{2(d+1)}{d+3}}(I)} \le \epsilon.
\end{align*}
\end{enumerate}

Then there exists $\epsilon_0 = \epsilon_0(d,E)$ such that if
$0<\epsilon<\epsilon_0$, for (NLW) with initial data $(u_0,u_1)$, there exists
a unique solution $u$ on $I\times \R^d$ with the properties
\begin{align}
 \| \tilde u - u\|_{\dot S^1(I)} & \le C(d,E) \cdot \epsilon^{c_3}, \notag\\
\|\tilde  u- u\|_{S(I)} & \le C(d,E) \cdot \epsilon^{c_4}.
\end{align}
Here $0<c_3<1$ is a constant depending on $(d,E)$ and $0<c_4<1$ depends only on the
dimension $d$.
\end{cor}

\end{document}